\documentclass[10pt]{article}
\usepackage[a4paper, total={6in, 9in}]{geometry}
\usepackage{graphicx}
\usepackage{amssymb}
\usepackage{amsmath}
\usepackage{amsthm}
\usepackage{algorithm}
\usepackage[noend]{algpseudocode}

 \newtheorem{problem}{\bf Problem}
 \newtheorem{theorem}{\bf Theorem}
 \newtheorem{definition}{\textbf{Definition}}
 \newtheorem{result}{\bf Result}
 \newtheorem{lemma}{\bf Lemma}
 \newtheorem{example}{\bf Example}

 \title{\textbf{Solution of network localization problem with noisy distances and its convergence}}
 \author{Ananya Saha\hspace{1cm} Buddhadeb Sau}
 \date{}

\begin{document}
\maketitle 
\graphicspath{{figures/},{simulations/}}

\begin{abstract}
The \textit{network localization problem} with convex and non-convex distance constraints may be 
modeled as a nonlinear optimization problem. The existing localization techniques are mainly based 
on convex optimization. In those techniques, the non-convex distance constraints are either ignored 
or relaxed into convex constraints for using the convex optimization methods like SDP, least square 
approximation, etc.. We propose a method to solve the nonlinear non-convex network localization 
problem with noisy distance measurements without any modification of constraints in the general 
model. We use the nonlinear Lagrangian technique for non-convex optimization to convert the problem 
to a root finding problem of a single variable continuous function. This problem is then solved 
using an iterative method. However, in each step of the iteration the computation of the functional 
value involves a finite mini-max problem (FMX). We use smoothing gradient method to fix the FMX 
problem. We also prove that the solution obtained from the proposed iterative method converges to 
the actual solution of the general localization problem. The proposed method obtains the solutions 
with a desired label of accuracy in real time. 
\end{abstract}


\section{Introduction}
\label{S1}
In recent technological advances, sensor networks are being adopted for collecting data from 
different hostile environments and monitoring them (Figure~\ref{fig:network}). A network may 
consists of sensor nodes, RFID readers, or members in a rescue team in a disaster management system, 
etc.
\begin{figure}[ht]
  \centering
  \includegraphics[width=3.5in]{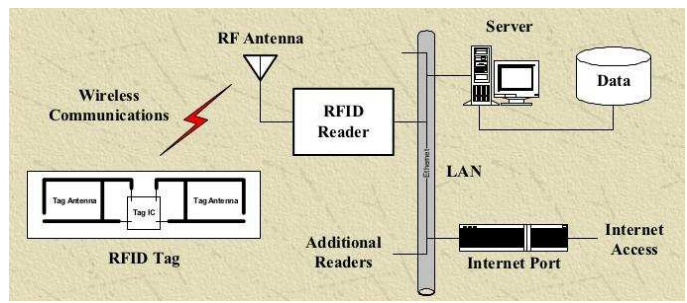}
  \caption{How does a network works}
    \label{fig:network}
\end{figure}
Air pollution monitoring~\cite{KRA10}, forest fire detection~\cite{MM07}, landslide 
detection~\cite{MVR9}, water quality detection~\cite{FDM5}, natural disasters prevention are some 
familiar field of applications in which sensor networks are useful. When a network is deployed in 
some region, the sensor nodes identify the events within their sensing ranges and transmit the 
collected information to the nodes within their communication ranges (a node within communication 
range of another node is called a \textit{neighbor}). The location of an event can naively be 
estimated by the positions of the nodes identifying the event. Thus knowing the locations of the 
nodes are essential for properly monitoring the events. The objective of the \textit{network 
localization} is to determine the node locations of a network using available distance information.

The \textit{GPS} (Global Positioning System)~\cite{BFJ12} installation with each node of a network 
for finding its location is costly. Therefore the localization technique without using GPS needs 
research focus. Many researchers have proposed novel algorithms so far, for finding node positions 
using the information available from neighboring nodes. Distance measurements among neighboring 
nodes are popularly used for computing the node positions. These distances are measured by 
instruments embedded inside the nodes. It is practically difficult to exactly measure the distances 
among nodes even with existing sophisticated hard-wares. Thus the network localization with noisy 
distance measurements demands rigorous research. In the literature, there are several algorithms 
based on exact~\cite{JS79, AGY4} as well as noisy distances~\cite{DPG1, PKY8}. Graph rigidity 
theory~\cite{JS79} and optimization theory are popularly used by researchers for developing 
localization algorithms~\cite{DPG1, PKY8}.

Any network may be represented by a distance graph (a graph with edge weights equal to the distances 
between the end points of the edge). The network localization problem with the graph model of the 
network is equivalent to the graph realization problem. In $1979$, Saxe~\cite{JS79} proved that the 
problem of embedding graphs in the one dimensional space is NP-complete and in higher dimensional 
spaces it is NP-hard. Later Aspnes, Goldenberg and Yang~\cite{AGY4} proved that the problem of 
finding a realization of a graph is an NP-hard problem even if it is known that the graph has unique 
realization.

During the last few decades, some variants of the general network localization problem have been 
solved. In~\cite{SM13sp, SM09, SMM06} the localization problem with exact node distances has been 
discussed for wireless sensor networks; they used the ordering of nodes of the underlying network 
and graph rigidity property for localization. In distance-based network localization, the number of 
solutions of the network localization problem may be unique, finite or infinite (up to congruence). 
Testing the unique localizability of networks having exact distances among nodes has been discussed 
in~\cite{JJ05, AE06}. If a network is not uniquely localizable then it must have some nodes which 
may either be freely rotated with respect to some other nodes or reflected with respect to some 
edges. If some nodes of the network may be rotated then the number of solutions of the associated 
localization problem is infinite. If in the underlying graph of a network, a vertex (or a few 
vertices) may be reflected with respect to a set of neighbors that are almost co-linear then it is 
called a \textit{flip vertex} of the network and this phenomenon in network localization is called a 
\textit{flip ambiguity}. Analysis of flip ambiguity in network localization has been discussed 
in~\cite{AB10}. To find unique localization of networks removing the flip ambiguity of nodes is 
essential. 

In real field of applications, collecting the exact distances among adjacent pair of nodes is almost 
impossible. Doherty et al.~\cite{DPG1} formulated the localization problem with noisy distance as a 
non-convex optimization problem. They excluded the non-convex constraints from the general problem 
to obtain a convex version of it and solved the problem using semi-definite programming 
(SDP)~\cite{DPG1}. Biswas et al.~\cite{PKY8} converted the same non-convex network localization 
problem to convex localization problem by using the relaxation technique and solved by SDP. To the 
best of our knowledge, solving the general problem is still a challenge. In this work we sort the 
challenge by solving the general nonlinear non-convex network localization problem using Lagrangian 
optimization. As far as we know this is the first approach for solving the network localization 
problem without any modification of constraints.

In our previous works~\cite{AB16,SS17} (published in the proceedings of international conferences), 
we converted the nonlinear non-convex network localization problem to a root finding problem of a 
single variable continuous function $\phi(c_0)$, $c_0\in\mathbb R$ (where $\phi(c_0)=\psi(c_0)-1$). 
We choose the standard bisection method for solving this root finding problem since the iterations 
in this method are guaranteed to converge to a root. The root finding problem inherently includes 
the finite mini-max problem which is NP-hard~\cite{DP95}. We used the sequential quadratic 
programming~\cite{MG99, CZ10, PJD14} method to solve the finite mini-max problem. Using the 
sequential quadratic programming method, an approximate finite mini-max value of the function 
$\phi(c_0)$ was computed at $c_0$. Therefore in the iterative method, the sign of the function 
$\phi(c_0)$ was determined incorrectly due to the approximation. For instance, suppose at an 
iterative step, an approximate value of $\phi(c_0)$ is $.0055$. But the actual value of the 
function is $-.0000025$. In this case, the bisection method will consider the sign of the function 
positive though it is actually negative. Thus the correct interval containing the root was not 
determined in the bisection method due to approximation. This is a drawback of using the bisection 
method which is rectified in this paper. 
 
A network localization problem may have different congruent solutions in the Euclidean space even 
if the distance information are collected from some practical field of application of WSN. The 
Euclidean space is unbounded. Therefore the congruent solutions of the localization problem may be 
anywhere in the Euclidean space. In our work, we constructed the root finding 
problem~\cite{AB16,SS17} such that the estimated node positions of the network will be closer to 
the origin of the Euclidean Space with respect to some rectangular axis. We established that, we 
may always identify a compact region (i.e., a closed and bounded region) in the two dimensional 
Euclidean space containing the origin within which the localization problem must have a 
solution~\cite{AB16,SS17}.  

In this paper, the construction of the root finding problem from the network localization problem 
is revised thoroughly. We develop an iterative method in light of the bisection method for finding 
a root of $\phi(c_0)$, $(c_0\in\mathbb R)$. In each step of the iterative method it is required to 
determine whether the function $\phi(c_0)$ has values with opposite signs at the end points of a 
sub-interval of the interval identified in the previous iteration. We compute a tight bound for the 
function $\phi(c_0)$ at $c_0$ which depends on its approximately computed value in the iteration. 
Using these bounds and the monotonic non-increasing property of $\phi(c_0)$ we determine the 
required sub-interval in our method. In this way without computing the exact value of $\phi(c_0)$ 
we proceed for finding a solution of $\phi(c_0)$. We establish that the method converges to a 
solution of the network localization problem and the solutions of root finding problem may be 
achieved up to a desired label of accuracy within an acceptable number of iterations.

\textit{Organization of the paper}: In Section~\ref{S2}, we present the general network 
localization problem with the convex and non-convex distance constraints. The construction of 
Lagrangian form of the network localization problem is given in Section~\ref{S3}. In 
Section~\ref{S4}, we discuss the technique for solving the Lagrangian network localization problem. 
Convergence of the solution technique is analyzed in Section~\ref{S5} along with some instances of 
networks for which we implement the root finding method for finding a localization. We sketch an 
error analysis of the proposed method in Section~\ref{S6} and conclude in Section~\ref{S7}.  

\section{Network localization problem}
\label{S2}

Let $\aleph$ be an ad-hoc network; $V$ is the set of nodes ($|V|=\nu$) and $E$ is the set of 
communication links. The underlying graph $G_{\aleph}(V,E)$ is the \textit{grounded graph} of 
$\aleph$. A \textit{realization} of $G_{\aleph}(V,E)$ in $d$-space is a $1$-$1$ mapping $f$ from 
the 
vertex set $V$ to $\mathbb R^{d}$. Two different realizations $f$ and $g$ of $G_{\aleph}$ are 
\textit{equivalent} if for each edge $\{u,v\}$ in $E$, $||f(u)-f(v)||=||g(u)-g(v)||$ where, $||.||$ 
is the standard Euclidean norm in $\mathbb R^{d}$. $f$ and $g$ are \textit{congruent} if the 
equality $||f(u)-f(v)||=||g(u)-g(v)||$ holds for each pair of vertices in $V$. A realization $f$ of 
$G_{\aleph}(V,E)$ in $d$-space is \textit{unique} up to congruence if every realization $g$ 
equivalent to $f$ is congruent to $f$. If $f$ and $g$ are congruent to each other then $g$ can be 
obtained from $f$ by a suitable transformation of the coordinate system in $d$-space and 
recomputing 
$f$ according to the new coordinate system. To fix a coordinate system in a $d$-space, $d+1$ 
independent points are required with known positions. Therefore a uniquely realizable framework can 
be uniquely located in a $d$-dimensional space if we can fix $d+1$ points in the space. On the 
other 
hand, if $G_{\aleph}$ has two or more equivalent realizations which are non-congruent in a 
$d$-space 
then $G_{\aleph}$ is called \textit{ambiguously} $d$-realizable.

A distance based localization algorithm determines the locations of nodes in a network by using 
known positions of anchors, if any, and a given set of inter-node distance measurements. Let 
$A=\{u_1,u_2,\cdots,u_m\}$ be the the set of anchor nodes with known positions $\{a_{1}$, $a_{2}$, 
$\cdots$, $a_{m}\}$ and $B=\{b_{1},b_{2}, \cdots, b_{n}\}$ be the nodes with unknown positions 
$\{x_{1}, x_{2}, \cdots, x_{n}\}$. In this work, we find the positions of the nodes in $B$ assuming 
$m=0$, i.e., the network has no anchor node. The technique is equally applicable for networks with 
anchor vertices. Let $\emph N$ be the set of all edges joining $x_i$'s. Upper and lower bound on 
the 
exact length of an edge in $\emph N$ joining $x_i$ and $x_j$ are denoted by $\overline{d}_{ij}$ and 
$\underline{d}_{ij}$. Let $\underline{D}=(\underline{d}_{ij})_{n\times n }$ and 
$\overline{D}=(\overline{d}_{ij})_{n\times n }$ be the matrices with $ij$-th entries 
$(\underline{d}_{ij})$ and $(\overline{d}_{ij})$ respectively. If two nodes $x_{i}$ and $x_{j}$ are 
not adjacent in the grounded graph $G_{\aleph}(V,E)$ then both the matrices have $ij$-th entry 
zero. 
Under the anchor free setting, the network localization problem may also be formulated as a 
nonlinear problem. The problem may formally be described as follows.

\begin{problem}
Given the edge set $N$ and the matrices $\underline{D}=(\underline{d}_{ij})_{n\times n}$,
$\overline{D}=(\overline{d}_{ij})_{n\times n}$, of the grounded graph $G_{\aleph}(V,E)$ of a 
network $\aleph$ with a set $B$ of $n$ nodes with unknown positions,
\begin{center}
\begin{tabular}{lll}
 Find & $X=(x_1, x_{2}, \cdots, x_{n})$   & \\
 such that & $\underline{d}_{ij}^{2} \leq ||x_{i}-x_{j}||^{2} \leq \overline{d}_{ij}^{2}$,
                                  $\{b_{i},b_{j}\}\in N$
\end{tabular}
\end{center}
\label{p1}
\end{problem}

If $X'$ is an estimation for the unknown positions of nodes in $B$ obtained by solving 
Problem~\ref{p1}, every realization congruent to $X'$ obtained by translating the coordinate system 
is also a solution to Problem~\ref{p1}. The re-computations of such solutions can be avoided by 
including a function $\min\sum_{i=1}^{n}||x_{i}||^{2}$ in the Problem~\ref{p1} as the objective 
function. A solution to this optimization problem will minimize the sum of square distances of 
unknown nodes from the origin. 

Let each $x_{i}\in\mathbb R^{d}$ then $x=(x_1$,$x_2$,$\cdots$,$x_n)$ is a point in $\mathbb 
R^{dn}$. Suppose $|N|=n_{0}$. For each edge $e_{k}\in N$, let $f_{k}:\mathbb R^{dn}\rightarrow 
\mathbb R$ be the function
\begin{center}
 $f_{k}(x)=||x_{i}-x_{j}||^{2},~~~~~1\leq k\leq n_{0} $
\end{center}
where $e_k=\{b_i,b_j\}$. Let $\underline{d}_k=\underline{d}_{ij}$ and 
$\overline{d}_k=\overline{d}_{ij}$. Using these notations, Problem~\ref{p1} can be rewritten as
Problem~\ref{P2}.

\begin{problem}
Given the matrices $\underline{D}=(\underline{d}_{k})_{n\times n}$,
$\overline{D}=(\overline{d}_{k})_{n\times n}$ of the network $\aleph$ with a set $X$ of $n$
nodes with unknown positions. Find solutions of the nonlinear optimization problem
\label{P2}
\begin{center}
\begin{tabular}{lll}
  Minimize  & $f_{0}(x)=\sum_{i=1}^{n}||x_{i}||^{2}$  &  \\[2.5mm]
  such that & $\underline{d}_{k}^{2} \leq f_{k}(x) \leq \overline{d}_{k}^{2}$,
                                         & $\forall~1\leq k\leq n_{0}$         
\end{tabular}
\end{center}
\end{problem}

In Problem~\ref{P2}, each constraint $\underline{d}_{k}^{2} \leq f_{k}(x) \leq 
\overline{d}_{k}^{2}$ 
can be broken into two parts, namely, $f_{k}(x)\leq\overline {d}_{k}^{2}$ and 
$\underline{d}_{k}^{2}\leq f_{k}(x)$. For each $k(1\leq k\leq n_0)$, if $y_1, y_2 \in \mathbb 
R^{dn}$ satisfy $f_{k}(x)\leq\overline {d}_{k}^{2}$ then $\forall t\in[0,1]$ 
$$f_{k}(ty_1+(1-t)y_{2})\leq tf_{k}(y_1)+(1-t)f_{k}(y_{2})\leq \overline d_{k}^2.$$ Therefore, each 
$f_{k}(x)\leq\overline {d}_{k}^{2}$ is a convex constraint~\cite{}. It can be shown that the 
constraints $\underline{d}_{k}^{2}\leq f_{k}(x)$ are not convex. Thus the constraints in 
Problem~\ref{P2} can be classified into two types based on the convexity,
\begin{equation}
\label{c1}
\textit{Convex constraints:~~~~~~~~~~~~~~}f_{k}(x)\leq \overline {d}_{k}^{2},
\end{equation}
\begin{equation}
\label{c2}
\textit{Non-convex constraints:~~~~~~~~}\underline{d}_{k}^{2}\leq f_{k}(x).
\end{equation}

This work is focused on solving the network localization problem keeping the non-convex distance 
constraints unaltered. Though Doherty, et al.~\cite{DPG1} formulated the localization problem as 
non-convex optimization problem~\cite{DPG1} they exclude the non-convex distance constraints to 
solve the problem using semi-definite programming (SDP). Biswas, et al.~\cite{PKY8} converted the 
non-convex network localization problem into a convex optimization problem by relaxing the 
non-convex inequality constraints and solved the relaxed problem~\cite{TM97, FK97} using 
SDP~\cite{VB96}. A reason behind using SDP method is that the SDP is approximately solvable in 
polynomial time~\cite{BJ12}. Yet none of these approaches solved the general network localization 
problem.

In this paper, using the Lagrangian theory, the anchor free network localization problem with noisy 
distance measurements is converted into a root finding problem without any modification of the 
nonlinear non-convex distance constraints. We solve the root finding problem using an iterative 
method and prove the convergence of the method to a solution of the localization problem. The 
method gives an estimation for node positions up to a desired level of accuracy within a real time 
period.

\section{Root finding problem construction using Lagrangian function}
\label{S3}
The network localization problem is inherently a non-convex optimization problem. In this section, 
we describe the Lagrangian function with the help of which we transform the general localization 
problem into a root finding problem. In Problem~\ref{P2}, each non-convex constraint 
$\underline{d}_{k}^{2}\leq f_{k}(x)$ can be written as $f_{n_0+k}(x) \leq \underline {d}_{k}^{2}$ 
where, $2\underline d_{k}^2 - f_k(x) = f_{n_0+k}(x)$. These modifications convert Problem~\ref{P2} 
into a non-convex optimization problem as described in Problem~\ref{P3}.

\begin{problem}
Given the matrices $\underline{D}=(\underline{d}_{ij})_{n\times n}$,
$\overline{D}=(\overline{d}_{ij})_{n\times n}$ of the network $\aleph$ consisting of $n$ nodes
with unknown positions $x = (x_1, x_{2}, \cdots, x_{n})$, solve the nonlinear optimization problem:
\label{P3}
\begin{center}
\begin{tabular}{lll}
 Minimize  & $f_0(x)$                                                                  \\
 such that & $f_{k}(x)\leq c_k$,
                                       & $1\leq k\leq 2n_{0}=r$(say)                 
\end{tabular}
\end{center}
where, $c_k=\overline d_k^2,~1\leq k\leq n_0$ and $c_k=\underline d_{k-n_0}^2,~n_0+1\leq k\leq r$.
\end{problem}

\subsection{Lagrangian function}
Let $c=[c_{0},c_{1},c_{2},\cdots,c_{r}]$, where $c_{0}$ is a positive real number independent of 
$x$ and $c_k(1\leq k\leq r)$ are defined in Problem~\ref{P3}. Note that, $c_k>0~\forall k$, because 
the distance information for each pair of nodes is collected from a network where no two sensors 
are in the same position.

\begin{definition}
  The Lagrangian function for Problem~\ref{P3} is defined as,
  $$\mathbb{L}(x,c_0)=\displaystyle\max_{0\leq k\leq r}\frac{f_{k}(x)}{c_{k}}.$$
\end{definition}
Lagrangian function may be defined in many ways~\cite{PKY8} for an optimization problem among which 
we consider the above form for the Lagrangian function in this paper. Shortly we prove that the
Lagrangian function always attains its infimum within the field of interest. With the help of 
$\mathbb L(x,c_0)$ the problem defined in Problem~\ref{prob:lagOpt} is later proved to be 
equivalent to \ref{P3} under certain restrictions which are acceptable in any real situations.

\begin{problem} Let the function $\mathbb L(x,c_0)$ attains its infimum at some point $z$ over the
domain of definition, i.e.,
\label{prob:lagOpt}
 $$\mathbb{L}(z,c_0)=
    \displaystyle\inf_{x\in\mathbb R^{dn}}\displaystyle \max_{0\leq k\leq r}
                                                        \frac{f_{k}(x)}{c_{k}}.$$\
We have to find the $z$.
\end{problem}

Below we describe a result from~\cite{GY01} which says that, $x$ is an optimal solution of 
Problem~\ref{P3} if and only if it is a solution of Problem~\ref{prob:lagOpt}. Thus if we can  
find a solution $x$ of Problem~\ref{prob:lagOpt} then $x$ may easily be mapped to an optimal 
solution of Problem~\ref{P3} using this result. It may also be noted that this technique does not 
need the convexity of the constraint functions, i.e. we do not ignore the non-convex constraints 
from the general problem.

\begin{result}[\cite{GY01}]
\label{r1}
  Let $\overline x\in\mathbb R^{dn}$ be an optimal solution of the network localization problem as
  defined by Problem~\ref{P3} and $c_{0}=f_{0}(\overline{x})>0$. A different $x_{0}\in\mathbb
  R^{dn}$ is an optimal solution of Problem~\ref{P3} if and only if $x_{0}$ is a solution of the
  unconstrained problem defined in Problem~\ref{prob:lagOpt}.
\end{result}

In the rest of this section, using Lagrangian theory for non-convex optimization
problem~\cite{GY01}, we convert Problem~\ref{P3} into a root finding problem involving single
variable.

\subsection{Lagrange's optimization problem}
\label{ss1}
The network models under consideration are picked up from networks already embedded in the field of
interest. For such an already embedded network, Problem~\ref{P3} satisfies the following conditions:
\begin{enumerate}
 \item The problem always has at least one feasible solution, since the graph underlying the 
network 
is constructed from a network already embedded in the field of interest.
 \item In Problem~\ref{P3}, $\displaystyle\lim_{||x||\rightarrow\infty} f_0(x)=\infty$. Since in  
practical applications, the field of interest is always bounded, the feasible region of       
Problem~\ref{P3} is also bounded. Since the objective function     
$f_0(x)=\sum_{i=1}^{n}||x_i||^2$ is everywhere continuous, there always exists some real constant 
$M$ such that $0<f_0(x)\leq M$ for all feasible $x$.
 \item Since $f_0(x)\leq M$ in the feasible region and Problem~\ref{P3} has feasible solution,      
 
Problem~\ref{P3} always possesses an optimal solution, say $\overline{x}$, in $\mathbb R^{dn}$. At 
$x=\overline{x}$, $f_0(\overline{x})\geq 0$. It may be noted that $f_0(\overline{x})=0$ only when 
all the points are at origin.
 \item In Problem~\ref{P3} the feasible region $X_0=\{x\in\mathbb R^{dn}/f_{k}(x)\leq c_k,~1\leq 
k\leq r\}$ is compact. It is shortly proved in Lemma~\ref{l1}.
 \item Since $f_0$ is polynomial, it is uniformly continuous on the feasible region $X_0$.
 \item For $c_0\geq f_0(\overline{x})$, there always exists some $x\in\mathbb R^{dn}$ such that
$f_0(\overline{x})\leq f_0(x)\leq c_0$ since $f_0$ is continuous.
\end{enumerate}

Under the above assumptions, we develop the following result which is used for developing the
proposed localization problem.

\begin{lemma}
\label{l1}
In Problem~\ref{P3} the feasible region $$X_0=\{x=(x_1,x_2,\ldots,x_n)\in\mathbb R^{dn} 
\Big | f_{k}(x)\leq c_k,~1\leq k\leq r\}$$ is compact.
\end{lemma}

\begin{proof}
Without lose of generality we restrict $X_0$ in $\mathbb R^{dn}_+$. Otherwise the origin may be
shifted so that the feasible region is included in $\mathbb R^{dn}_+$. 

A set is compact in $\mathbb R^{dn}$ if and only if it is both bounded and closed~\cite{MAA}. We 
give an explicit proof of the compactness by showing that the above defined set is both bounded and 
closed in $\mathbb R^{dn}$. In view of the condition $(2)$, $X_0$ is bounded.

\textit{Closed-ness of $X_0$:} Let for an arbitrarily chosen $k$, $X_k=\{x\in\mathbb 
R^{dn}/f_{k}(x)\leq c_k\}$ and $\{y^l\}_l$ [where $y^l=(y^l_1,y^l_2,\ldots,y^l_{n})$] be a Cauchy 
sequence in $X_k$ with limit $x$. Let $k$-th edge of the grounded graph joins the nodes $i,j$ of 
the network. For $1\leq k\leq n_0$, $f_{k}(y^l)\leq c_k$ $\Rightarrow$ $||y^l_i-y^l_j||^2\leq c_k$ 
and if $n_0+1\leq k\leq 2n_0$, $f_{k}(y^l)\leq c_k$ $\Rightarrow$ $2c_k-||y^l_i-y^l_j||^2\leq c_k$. 
Let us first consider the case $1\leq k\leq n_0$.

Since $y^l\rightarrow x$ therefore for given any $\varepsilon>0$ there exists some $m\in\mathbb
N$ where for all $l\geq m$
                   $$||y^l-x||<\varepsilon.$$
This gives for all $l\geq m$,
                   $$||y^l_u-x_u||<\varepsilon ~~~~~~for~each~ 1\leq u\leq n.$$
Thus for $l\geq m$,
\begin{center}
\begin{tabular}{ll}
        & $||x_i-x_j||$\\
 $=$    & $||x_i-y^l_i+y^l_i-y^l_j+y^l_j-x_j||$\\
 $\leq$ & $||x_i-y^l_i||+||y^l_i-y^l_j||+||y^l_j-x_j||$\\
 $<$    & $2\varepsilon+|\sqrt c_k|.$
\end{tabular}
\end{center}

Since the above relation holds for arbitrarily chosen $\varepsilon$ we get $||x_i-x_j||\leq|\sqrt 
c_k|$, i.e., $||x_i-x_j||^2\leq c_k$. Thus if $f_k(y^l)=||y^l_i-y^l_j||^2$ then $X_k$ is a compact 
set. For $n_0+1\leq k\leq 2n_0$, the proof for closed-ness is similar as before. Hence for each 
$k$, $X_k$ is closed. $X_0$ is the intersection of finite number of closed sets and it is closed. 
Therefore $X_0$ is compact.
\end{proof}

The optimum solution of $\displaystyle\inf_{x\in\mathbb R^{dn}}\displaystyle\max_{0\leq k\leq
r}f_{k}(x)/c_{k}$ varies for different values of $c_0$ (for $1\leq k\leq r$, each $c_k$ is given 
in the problem as constants). We define a scalar function $\psi$ of parameter $c_{0}\in\mathbb
R_{+}-\{0\}$ as follows,
\begin{equation}
\label{e1}
\psi(c_{0})=\displaystyle\inf_{x\in\mathbb R^{dn}}\displaystyle \max_{0\leq k\leq r}
 \{f_{k}(x)/c_{k}\}.
\end{equation}

To construct the Lagrangian optimization problem we here present some results from~\cite{GY01} 
involving function $\psi(c_{0})$.

\begin{result}[\cite{GY01}] \label{r2}
Let $\{f_k:(0\leq k\leq r)\}$ be a finite set of continuously differentiable functions defined on an
unbounded set $X$, $c_k> 0$. Consider the optimization problem
\begin{center}
\begin{tabular}{ll}
 $\min f_0(x)$\\
 $subject~to~f_{k}(x)\leq c_k~(1\leq k\leq r)$
\end{tabular}
\end{center}
under the following assumptions:
\begin{enumerate}
 \item The feasible region is compact.
 \item $\displaystyle\lim_{||x||\rightarrow\infty}f_0(x)=\infty.$
 \item If $\overline{x}$ is an optimal solution and for any arbitrary constant 
$c_0>f_0(\overline{x})$
\end{enumerate}
Then there exists some $x\in X$ such that $f_0(\overline{x})\leq f_0(x)< c_0$. Also the following
conditions hold.
\begin{enumerate}
 \item $c_{0}< f_{0}(\overline x)\Rightarrow\psi(c_{0})> 1$.
 \item With addition to the above conditions if $f_0(x)$ is uniformly continuous then $c_{0}\geq
       f_{0}(\overline x)\Rightarrow\psi(c_{0})\leq 1$.
 \item $\psi$ is a monotone non-increasing continuous function.
 \item $c_{0}=f_{0}(\overline{x})$ if and only if $\psi(c_{0})=1.$
\end{enumerate}
\end{result}

\begin{theorem} Let $\overline{x}$ be an optimal solution of Problem~\ref{P3}. Then $\psi(c_0)$ has 
the following properties:
\label{t1}
\begin{enumerate}
    \item If $c_0< f_0(\overline x)$ then $\psi(c_0)> 1$.
    \item If $c_{0}\geq f_{0}(\overline x)$ then $\psi(c_{0})\leq 1$.
    \item $\psi$ is a non-increasing continuous function of $c_{0}$.
    \item $c_{0}=f_{0}(\overline{x})$ if and only if $\psi(c_{0})=1.$
\end{enumerate}
\end{theorem}

\begin{proof}
Under the network model, we have seen that $f_i$s in Problem~\ref{P3} are continuously 
differentiable and $c_i>0$. The above mentioned condition~$(4)$ of the underlying network model 
shows that the feasible region of Problem~\ref{P3} is compact. The condition~$(3)$ shows that the 
Problem~\ref{P3} has an optimal solution, say $\overline{x}$. By the condition~$(6)$, there always 
exists some $x\in\mathbb R^{dn}$ such that $f_0(\overline{x})\leq f_0(x)\leq c_0$ when $c_0\geq 
f_0(\overline{x})$. The proof of this theorem then follows from Problem~\ref{r2}.
\end{proof}
In the following paragraph, we present the network localization problem as a root finding problem
which may be obtained from Problem~\ref{P2} using Theorem~\ref{t1}.

\begin{problem} Given the matrices $\underline{D}=(\underline{d}_{ij})_{n\times n}$,
$\overline{D}=(\overline{d}_{ij})_{n\times n}$ of the network $\aleph$ with a set $X$ of $n$
nodes with unknown positions. Let $\psi(c_0)=\displaystyle\inf_{x\in\mathbb
R^{dn}}\displaystyle \max_{0\leq k\leq r}\{f_{k}(x)/c_k\}.$
\begin{center}
\label{P6}
    Find $c_0$ such that $\psi(c_0)=1$.
\end{center}
\end{problem}

To get a good estimation for the node positions in the network, we need to search for some positive 
real number $c$ for which there exists some $x\in\mathbb R^{dn}$ such that the value of the 
function 
$\psi$ is equal or very close to $1$. In the rest of this paper, we will concentrate for finding or 
estimating the roots of $\psi(c_0)=1$.

\section{Solving the root finding problem}
\label{S4}

In the previous section, we have seen that solving the network localization problem is equivalent 
to 
solving Problem~\ref{P6}. Here we prove that if we can find a root $c_0=c^*_0$ of the equation 
$\psi(c_0)=1$, node positions of the network will be obtained from the corresponding $x$ at which 
$\mathbb L(x,c_0)$ $(=\displaystyle \max_{0\leq k\leq r}\{f_{k}(x)/c_{k}\})$ exactly equals $1$.
\begin{theorem}
 Let $c_0^*$ is a root of the equation $\psi(c_0)=1$. Then there exists an optimizing $x=x^*$ such
 that, $$\psi(c_0^*)=\mathbb L(x^*,c_0^*)=1.$$
\end{theorem}

\begin{proof}
Since the feasible region of the network localization problem is compact the optimal solution of 
the 
problem lies within a compact set. Therefore instead of searching the minimizing $x$ of the 
function 
$\displaystyle\max_{0\leq k\leq r,c_0=c_0^*}\{f_k(x)/c_k\}$ all over $\mathbb R^{dn}$ we may 
restrict our search on a compact subset, say $K\subset\mathbb R^{nd}$, containing the feasible 
region of the network localization problem. Such a compact set for Problem~\ref{P6} may be 
constructed as follows:

Consider the field of interest, $F$, for localizing the network in $\mathbb R^d$. If $x$ is a 
feasible solution of the general network localization problem somewhere in $\mathbb R^{d}$ then by
using the translation and rotation operations we may get a congruent realization of the network in
$F$. Since $F$ is bounded we will get some upper bound as well as lower bound for each coordinate
of any point lying in the region. If we get any localization of the network obtained by solving
Problem~\ref{P6} then it will lie within the field of interest.

Let $M'$ and $m'$ be the maximum and minimum for all of the $d$ coordinates in the field of
interest. Consider a $dn$-dimensional box 
 $$K=\{(y_1,y_2,\ldots,y_i,\ldots,y_{dn})|~m'\leq y_i\leq M',\forall i\}$$ 
in $\mathbb R^{dn}$. $K$ is always compact. Let $x=(x_1,~\ldots,~x_n)$ be a realization of the 
network within the field of interest. For each $i$, if $x_i=(x_{i1},x_{i2},\ldots,x_{id})$ then 
$m'\leq x_{ij}\leq M'$. Therefore corresponding to each solution of the network localization 
problem there is a point in the $dn$-dimensional box.

The function $\mathbb L(x,c_0^*)$ is a continuous function of the variable $x$ on this compact set
$K$. The proof of the theorem will be followed if we can show that the continuous function $\mathbb
L(x,c_0^*)$ defined on $K$ attains its minimum at some point in $K$.

Since $K$ is a compact set and the function $\mathbb L$ is continuous, $\mathbb L(K,c_0^*)$ is a
compact set (i.e., closed and bounded). Also the infimum of any set is either a limit point or an 
element of the set. In both cases, the infimum of $\mathbb L(K,c_0^*)$ lies inside $\mathbb 
L(K,c_0^*)$ since, $\mathbb L(K,c_0^*)$ is closed. Therefore, we get some $x^*$ such that
$\mathbb L(x^*,c_0^*)= \displaystyle\inf_{x\in K}\mathbb L(x,c_0^*)$ i.e., 
$\psi(c_0^*)=\mathbb L(x^*,c_0^*)=1.$
\end{proof}

We develop an iterative method in light of the bisection method for finding a root of 
$\psi(c_0)=1$, $(c_0\in\mathbb R)$. The method is guaranteed to converge to a root of the 
continuous function $\psi(c_0)-1$ on an interval $[c_{01},c_{02}]$, if $(\psi(c_{01})-1)$ and 
$(\psi(c_{02})-1)$ have opposite signs. At the initial stage of the iterative method we search for 
an interval containing $c_0$ within which $\psi(c_0)-1$ must have a root. The searching process may 
progress as follows: 

Consider a real number $c_0=c_{01}>0$. For computing $\psi(c_{01})-1$ it is required to solve a 
finite mini-max problem which is an NP-hard problem~\cite{DP95}. We use smoothing gradient 
technique (Section~\ref{SS1}) for computing an approximate value of $\psi(c_{01})-1$. In this 
smoothing technique, $\psi(c_{01})-1$ may be approximated such that depending on the approximated 
value of the function we will get an interval within which the actual functional value lies. Using 
these bounds and the monotonic non-increasing property of the function we determine the sign of 
$\psi(c_{01})-1$ (Theorem~\ref{T1}). If $c_{01}$ is not a root of $\psi(c_0)-1$ then one of the 
following cases may occur.
\vspace{1mm}

\textbf{case $1$.} ($\psi(c_{01})-1$ is positive): Choose a point $c_0=c_{02}$ 
$(=c_{01}+\alpha,~\alpha$ is an arbitrary positive number$)$. Since $\psi(c_0)-1$ is a 
non-increasing continuous function of $c_0$ (Theorem~\ref{t1}), then for sufficiently large 
constant $\alpha$ either $\psi(c_{02})-1=0$ or $\psi(c_{02})-1< 0$. If $\psi(c_{02})-1=0$ then 
$c_{02}$ is a root of the equation $\psi(c_0)=1$ and we are done. Otherwise the required interval 
is $[c_{01},c_{02}]$.

\textbf{case $2$.} ($\psi(c_{01})-1$ is negative at $c_{01}$): Choose a point 
$c_0=c_{02}~(0<c_{02}<c_{01})$. With similar reason as in case 1, either $\psi(c_{02})-1=0$ or 
$\psi(c_{02})-1>0$. If $\psi(c_{02})-1=0$ then we are done, otherwise the required interval is 
$[c_{02},c_{01}]$.

In this way without computing the exact value of $\psi(c_0)-1$ we obtain an interval 
$[c_{01},c_{02}]$ at the end points of which $\psi(c_0)-1$ take values with opposite signs and 
proceed for finding a solution of $\psi(c_0)-1$. In the following section we describe the smoothing 
gradient technique which we implemented for approximately computing $\psi(c_0)-1$.

\subsection{Smoothing Gradient Technique for solving finite mini-max}
\label{SS1}

In the literature there are several smoothing techniques which may be used for solving a finite 
mini-max problem. Among those techniques we choose one for solving our finite mini-max optimization 
problem in which the function $\psi(c_0)-1$ remains bounded for each $c_0\in\mathbb R$. The 
technique uses a smoothing function (given in~(\ref{E})) to approximate the underlying non-smooth 
objective function $L(x,c_0)$. A smoothing function for a given non-smooth continuous function may 
be defined as follows.

\begin{definition}~\cite{PJD14}
 Let, $f:\mathbb R^n\rightarrow \mathbb R$ be a continuous non-smooth function. We call 
$\tilde{f}:{\mathbb R^{n}\times\mathbb R_+}\rightarrow\mathbb R$ a smoothing function of $f$ if 
$\tilde{f}(.,\mu)$ is continuously differntiable in $\mathbb R^{n}$ for every $\mu\in\mathbb R^+$ 
and $$\displaystyle\lim_{z\rightarrow x,\mu\downarrow 0}\tilde{f}(z,\mu)=f(x)$$
for any $x\in \mathbb R^{n}$.
\end{definition}

To solve the root finding problem we require an estimation for $\displaystyle\min_x L(x,c_0)$ 
$=$ $\displaystyle\min_x$ $\displaystyle\max_k\left\{\frac{f_k(x)}{c_k}\right\}$ $(c_0\in 
[c_{01},c_{02}])$, where $L(x,c_0)$ is a non-smooth continuous function. We consider the smoothing 
function for $\mathbb L(x,c_0)$ as follows:
\begin{equation}
\label{E}
\tilde{\mathbb L}(x,\mu,c_0) = 
\mu\log\displaystyle\sum_{k=1}^{r}\exp\left\{\frac{1}{\mu}\times\frac{f_k(x)}{c_k}\right\}. 
\end{equation}
The function $\tilde{\mathbb L}(x,\mu,c_0)$ provides a good estimation for $\mathbb L(x,c_0)$ since 
the following inequality holds.
\begin{theorem}
\label{T1}
$\mathbb L(x,c_0)\leq\tilde{\mathbb L}(x,\mu,c_0)\leq\mathbb L(x,c_0)+\mu\log r$, for 
$\mu>0$.
\end{theorem}

\begin{proof}~\\ 
\begin{tabular}{p{3.5cm}p{.5cm}p{9cm}}\multicolumn{3}{c}{~}\\
 $\tilde{\mathbb L}(x,\mu,c_0)-\mathbb L(x,c_0)$ & $=$ &
$\mu\log\displaystyle\sum_{k=1}^{r}\exp\left\{\frac{1}{\mu}\times\frac{f_k(x)}{c_k}\right\}$
$-$ $\mu\times\Big(\frac{1}{\mu}\times\mathbb L(x,c_0)\Big)$\\[1mm]
~ & $=$ & 
$\mu\log\displaystyle\sum_{k=1}^{r}\exp\left\{\frac{1}{\mu}\times\frac{f_k(x)}{c_k}\right\}$
$-$ $\mu\log\exp\Big(\frac{1}{\mu}\times\mathbb L(x,c_0)\Big)$\\[1mm]
~ & $=$ & 
$\mu\log\left\{\displaystyle\sum_{k=1}^{r}\exp\left\{\frac{1}{\mu}\times\frac{f_k(x)}{c_k}
\right\}/\exp\Big(\frac{1}{\mu}\times\mathbb L(x,c_0)\Big)\right\}$\\[1mm]
~ & $=$ & 
$\mu\log\left\{\displaystyle\sum_{k=1}^{r}\exp\left\{\frac{1}{\mu}\times\frac{f_k(x)}{c_k}
-\frac{1}{\mu} \times \mathbb L(x,c_0)\right\}\right\}$\\[1mm]
~ & $=$ & 
$\mu\log\left\{\displaystyle\sum_{k=1}^{r}\exp\Big(\frac{1}{\mu}\times\left\{\frac{f_k(x)}{c_k}
-\mathbb L(x,c_0)\right\}\Big)\right\}$
\end{tabular}

Since $\frac{f_k(x)}{c_k}\leq \mathbb L(x,c_0)=\displaystyle\max_{0\leq k\leq 
r}\frac{f_k(x)}{c_k}$, at any point $x\in \mathbb R^{nd}$,

$$\exp\Big(\frac{1}{\mu}\times\left\{\frac{f_k(x)}{c_k}-\mathbb L(x,c_0)\right\}\Big)\leq 1$$
for each $k$. This gives,
$$\displaystyle\sum_{k=0}^{r}\exp\Big(\frac{1}{\mu}\times\left\{\frac{f_k(x)}{c_k}
-\mathbb L(x,c_0)\right\}\Big)\leq r.$$
Hence, 
\begin{equation}
\tilde{\mathbb L}(x,\mu,c_0)-\mathbb L(x,c_0)\leq \mu \log r\Rightarrow
 \tilde{\mathbb L}(x,\mu,c_0)\leq\mathbb L(x,c_0)+\mu \log r. 
\label{eq11}
\end{equation}
Also for each $x\in \mathbb R^{nd}$ there is some $k=m$ for which 
$$\frac{f_m(x)}{c_m}=\mathbb L(x,c_0)=\displaystyle\max_{0\leq k\leq r}\frac{f_k(x)}{c_k}.$$
Thus we get, 
$$\exp\Big(\frac{1}{\mu}\times\left\{\frac{f_m(x)}{c_m}-\mathbb L(x,c_0)\right\}\Big)=1$$
which gives
$$\displaystyle\sum_{k=0}^{r}\exp\Big(\frac{1}{\mu}\times\left\{\frac{f_k(x)}{c_k}
-\mathbb L(x,c_0)\right\}\Big)\geq 1$$
or, $$\mu\log\Big(\displaystyle\sum_{k=0}^{r}\exp\Big(\frac{1}{\mu}\times\left\{\frac{f_k(x)}{c_k}
-\mathbb L(x,c_0)\right\}\Big)\Big)\geq 0$$
i.e., 
\begin{equation}
\tilde{\mathbb L}(x,\mu,c_0)\geq\mathbb L(x,c_0). 
\label{eq12}
\end{equation}
Combining Equation~\ref{eq11} and Equation~\ref{eq12} we get the given inequality.
\end{proof}

We describe the smoothing gradient algorithm below which will produce a clarkr stationary point  
(Appendix~\ref{app2}) $x_0$ for $\displaystyle\min_x L(x,c_0).$

\begin{algorithm}[H]
\begin{algorithmic}[1]
\caption{}
\Procedure{SmoothingGradientAlgorithm~\cite{PJD14}:}{}
\label{A1}
\State Choose $\sigma_1\in (0,0.5); \sigma_2\in (\sigma_1,1)$; $\gamma > 0; \gamma_1\in(0,1)$; 
$x_0\in \mathbb R^{nd}$; 
\State Let $i=0$; $\mu_0\gets$ A positive number chosen arbitrarily; $\epsilon\gets 10^{-4};$
\While {$\mu_i\geq \epsilon$}
  \State  $g_i\gets\triangledown\tilde{\mathbb L}(x_i,\mu_i,c_0)$;
  \State  $d_i\gets(-g_i)$;
  \State $\alpha\gets\Call{WolfLineSearchAlgorithm}{x_i,d_i,\mu_i,\sigma_1,\sigma_2}$
  \State  $x_{i+1}\gets x_i+\alpha*d_i$;
  \If {$||\triangledown\tilde{\mathbb L}(x_{i+1},\mu_i,c_0)||\geq\gamma*\mu_i $}
      \State $\mu_{i+1}\gets\mu_i;$
  \Else 
      \State $\mu_{i+1}\gets\gamma_1*\mu_i;$
  \EndIf
  \State  $i=i+1;$
\EndWhile
\EndProcedure
\end{algorithmic}
\end{algorithm}

In each step of the \Call{Smoothing Gradient Algorithm}{} we use 
\Call{WolfLineSearchAlgorithm}{}~\cite{url1} for finding $\alpha$ for the next iteration. The 
algorithm searches for finding the maximum value of the constant $\alpha$ satisfying the following 
two conditions. 
\begin{enumerate}
 \item $\tilde{\mathbb L}(x_i+\alpha d_i,\mu_i,c_0)\leq\tilde{\mathbb L}(x_i,\mu_i,c_0)
                           +\sigma_1\alpha g^T_i d_i$
 \item $\triangledown\tilde{\mathbb L}(x_i+\alpha d_i,\mu_i,c_0)^Td_i\geq\sigma_2 g^T_i d_i$
\end{enumerate}
where $x_i,d_i,\mu_i,\sigma_1,\sigma_2$ are from the smoothing gradient algorithm. The first 
condition ensures that at the $i+1$-th step of the iteration the functional value $\tilde{\mathbb 
L}(x_i+\alpha d_i,\mu_i,c_0)$ is smaller than $\tilde{\mathbb L}(x_i,\mu_i,c_0)$ 
(since $\sigma_1\alpha g^T_i d_i=-\sigma_1\alpha||g_i||^2$ is negative). Here we show that an 
$\alpha$ satisfying this condition always exists. From the \textit{Taylor theorem for multivariate 
functions}~\cite{MAA} of $\tilde{\mathbb L}(x_i+\alpha d_i,\mu_i,c_0)$ we get 
$$\tilde{\mathbb L}(x_i+\alpha d_i,\mu_i,c_0)=\tilde{\mathbb L}(x_i,\mu_i,c_0)
                           +\alpha g^T_i d_i+O(\alpha^2)~\cite{PJD14}$$
Therefore\\
\begin{tabular}{p{1cm}p{9cm}}\multicolumn{2}{c}{~}\\[1mm]
~  & $\tilde{\mathbb L}(x_i,\mu_i,c_0)+\alpha g^T_i d_i+O(\alpha^2)\leq \tilde{\mathbb 
L}(x_i,\mu_i,c_0)+\sigma_1\alpha g^T_i d_i$\\[1mm]
if, & $\alpha g^T_i d_i+O(\alpha^2)\leq\sigma_1\alpha g^T_i d_i$\\[1mm]
i.e. if, & $-(1-\sigma_1)\alpha||g_i||^2+O(\alpha^2)\leq 0,~or,(1-\sigma_1)\alpha||g_i||^2\geq 
O(\alpha^2).$
\end{tabular}
\vspace{2mm}

Using the Taylor's theorem it may be concluded that such an $\alpha$ always exists. Condition $(2)$ 
of \Call{WolfLineSearchAlgorithm}{} has been inserted to keep $\alpha$ sufficiently large as such 
the slope of $\tilde{\mathbb L}(x_i+\alpha d_i,\mu_i,c_0)$ remains at least $\sigma_2$ 
$(\sigma_1<\sigma_2<1)$ times larger than the slope of $\tilde{\mathbb L}(x_i,\mu_i,c_0)$ 
~\cite{url1}.                                
\begin{algorithm}[H]
\begin{algorithmic}[1]
\caption{}
\Procedure{WolfLineSearchAlgorithm($x_i,d_i,\mu_i,\sigma_1,\sigma_2$)~\cite{url1}:}{}
\label{p2}
\State Let $\alpha\gets0$; $t\gets 1$; $\beta\gets\infty$;
\State \textbf{Repeat}
\If {$\tilde{\mathbb L}(x_i+t d_i,\mu_i,c_0)>\tilde{\mathbb L}(x_i,\mu_i,c_0)
                           +\sigma_1 t g^T_i d_i$}
  \State Set $\beta\gets t;$ $t\gets\frac{1}{2}(\alpha+\beta)$;
\ElsIf{$\triangledown\tilde{\mathbb L}(x_i+t d_i,\mu_i,c_0)^Td_i<\sigma_2 g^T_i d_i$}
  \State $\alpha\gets t$; 
  \If {$\beta=\infty$} 
    \State $t\gets2\alpha$;
  \Else 
    \State $t=\frac{1}{2}(\alpha+\beta)$;
  \EndIf
\Else
  \State return $\alpha$;
\EndIf                           
\State\textbf{End Repeat}
\EndProcedure
\end{algorithmic}
\end{algorithm}

\section{Convergence analysis of the root finding method}
\label{S5}
The convergence of the root finding method inherently depends on the convergence of the 
\Call{SmoothingGradientAlgorithm}{}. 

\subsection{Convergence of SmoothingGradientAlgorithm}
Let $\{x_i\}$ and $\{\mu_i\}$ are the sequences generated by the smoothing gradient algorithm.
Towards proving the convergence of the method let us first consider the set 
                             $$S=\{i~|~\mu_{i+1}=\gamma_1\mu_i\}$$
in the smoothing gradient algorithm.                             
\begin{lemma}
 The set $S$ can not be finite.
 \label{L1}
\end{lemma}
\begin{proof}
If $S$ is a finite set then from the smoothing gradient algorithm we get, there exists an integer 
$i'$ such that for all $i>i'$, 
\begin{equation}
 ||\triangledown\tilde{\mathbb L}(x_{i+1},\mu_i,c_0)||\geq\gamma*\mu_i,
 \label{1}
\end{equation}
and $\mu_i=\mu_{i'}=\mu$ (say). If this is true then our claim is that,
\begin{equation}
 \displaystyle\liminf_{i\rightarrow\infty}||\triangledown\tilde{\mathbb L}(x_{i+1},\mu,c_0)||=0.
 \label{3}
\end{equation}
Suppose that~(\ref{3}) does not hold. Then there exists a sub-sequence of $\{g_i\}$ 
($g_i=\triangledown\tilde{\mathbb L}(x_{i+1},\mu,c_0)$), say $\{g_j\}$, for which
\begin{equation}
 ||g_j||\geq\epsilon~~for~~some~~\epsilon > 0~and~\forall~j.
 \label{4}
\end{equation}
But in Wolfe Line Search Algorithm $d_i$ $(=-g_i)$ is always a decent direction for the 
gradient~\cite{url1}. Thus the sequence $\Big\{\tilde{\mathbb L}(x_j, ~\mu, ~c_0)\Big\}$ generated 
by the algorithm is a monotonically decreasing sequence. Also $x_j\in K$ ($\forall$ $j$) where $K$ 
is the compact set (we have chosen in our paper). Therefore the sequence $\Big\{\tilde{\mathbb 
L}(x_j,~\mu,~c_0)\Big\}$ will be bounded below and convergent in $K$. Hence we get,
$$\tilde{\mathbb L}(x_j,\mu,c_0)-\tilde{\mathbb L}(x_{j+1},\mu,c_0)\rightarrow 0, 
~as~j\rightarrow\infty.$$
Using this and the condition $(1)$ of wolfe line search we obtain,
$$-g_j^T\alpha_j d_j\leq \frac{1}{\sigma_1}\Big(\tilde{\mathbb L}(x_j,\mu,c_0)-\tilde{\mathbb 
L}(x_{j+1},\mu,c_0)\Big)\rightarrow 0~~~~as~j\rightarrow\infty.$$
This gives $-g_j^T\alpha_j d_j\rightarrow 0~~as~~j\rightarrow\infty$. But $-g_j^T\alpha_j d_j= 
||g_j||.||\alpha_j d_j||\cos\theta$ where $\theta=\pi$ is the angle in between $g_j$ and $d_j$.
Thus we get 
\begin{equation}
||\alpha_j d_j||\rightarrow 0~since~||g_j||\geq\epsilon.
\label{5} 
\end{equation}
Since $x_{j+1}=x_j+\alpha_j d_j$ and $\triangledown\tilde{\mathbb L}$ differntiable on $K$, the 
Taylor's Theorem~\cite{MAA} on $\triangledown\tilde{\mathbb L}(x_{j+1},\mu,c_0)$ gives that,
 $$\triangledown\tilde{\mathbb L}(x_j+\alpha_j d_j,\mu,c_0)^T\alpha_j d_j
      =\triangledown\tilde{\mathbb L}(x_j,\mu,c_0)^T\alpha_j d_j
       +\alpha_j d_j\triangledown^2\tilde{\mathbb L}(x_j,\mu,c_0)^T\alpha_j d_j
        +o(||\alpha_j d_j||)$$
or,        
 $$\triangledown\tilde{\mathbb L}(x_j+\alpha_j d_j,\mu,c_0)^T\alpha_j d_j
      =\triangledown\tilde{\mathbb L}(x_j,\mu,c_0)^T\alpha_j d_j
       +o(||\alpha_j d_j||).$$
This gives, 
\begin{equation}
\displaystyle\lim_{j\rightarrow\infty}
\frac{\triangledown\tilde{\mathbb L}(x_{j+1},\mu,c_0)^T \alpha_j d_j} 
     {\triangledown\tilde{\mathbb L}(x_j,\mu,c_0)^T\alpha_j d_j}=1,
 \label{E2}
\end{equation}
since from~(\ref{5}) $o(||\alpha_j d_j||)\approx 0$ for $j\rightarrow\infty$. But~(\ref{E2}) gives 
a contradiction since from the second condition of wolfe line search algorithm we get the following

\begin{center}
\begin{tabular}{p{.8cm}p{2.9cm}p{.3cm}p{3cm}}\multicolumn{4}{c}{~}\\
~ & $\triangledown\tilde{\mathbb L}(x_{j+1},\mu,c_0)^T d_j$ & $\geq$ & $\sigma_2g_j^Td_j$\\[1.3mm]
or, & $-g(x_{j+1})^Tg(x_{j})$ & $\geq$ & $-\sigma_2 g(x_{j})^Tg(x_{j})$\\\\[.5mm]
or, & $g(x_{j+1})^Tg(x_{j})$ & $\leq$ & $\sigma_2 g(x_{j})^Tg(x_{j})$\\\\[1mm]
i.e., & $\frac{g(x_{j+1})^Tg(x_j)}{g(x_j)^Tg(x_j)}$ & $\leq$ & $\sigma_2<1$\\\\[1mm]
or, & $\frac{\triangledown\tilde{\mathbb L}(x_{j+1},\mu,c_0)^T \alpha_j d_j}
              {\triangledown\tilde{\mathbb L}(x_j,\mu,c_0)^T\alpha_j d_j}$ 
              & $\leq$ & $\sigma_2<1$
\end{tabular}
\end{center}
Therefore $\displaystyle\lim_{j\rightarrow\infty}\frac{\triangledown\tilde{\mathbb 
L}(x_{j+1},\mu,c_0)^T \alpha_j d_j}{\triangledown\tilde{\mathbb 
L}(x_j,\mu,c_0)^T\alpha_jd_j}\neq1$ and $||g_j||\geq\epsilon$ does not hold and we get,
\begin{equation}
 \displaystyle\liminf_{i\rightarrow\infty}||\triangledown\tilde{\mathbb L}(x_{i+1},\mu,c_0)||=0.
 \label{11} 
\end{equation}
But Equation~(\ref{11}) contradicts Equation~(\ref{1}) since $\gamma\mu$ is a constant. Thus $S$ 
can not be a finite set.  
\end{proof}

\begin{lemma}
 For the sequence $\{x_i\}$ and $\{\mu_i\}$ generated by the smoothing gradient algorithm the 
 following conditions hold, 
 \begin{enumerate}
  \item $\displaystyle\lim_{i\rightarrow\infty} \mu_i=0.$
  \item $\displaystyle\lim_{i\rightarrow\infty}||\triangledown\tilde{\mathbb L} 
        (x_{i+1},\mu_i,c_0)|| =0$.  
 \end{enumerate}
 \label{l2}
\end{lemma}

\begin{proof}
We use Lemma~\ref{L1} to prove both these results.
\begin{enumerate}
 \item Since $S$ is an infinite set $\mu_{i+1}=\mu_i\gamma_1$ $(\gamma_1<1)$ for infinitely many 
$i$-s in the smoothing gradient algorithm. Therefore $\displaystyle\lim_{i\rightarrow\infty} 
\mu_i=0.$  
 \item $\displaystyle\lim_{i\rightarrow\infty}||\triangledown\tilde{\mathbb L}(x_{i+1},\mu_i,c_0)|| 
\leq\gamma\displaystyle\lim_{i\rightarrow\infty}\mu_i=0.$
\end{enumerate}
\end{proof}
Since in Lemma~\ref{l2}, $\displaystyle\lim_{i\rightarrow\infty}\mu_i=0$ therefore 
$\displaystyle\lim_{\mu_i\downarrow 0}\tilde{\mathbb L}(x_{i+1},\mu_i,c_0)=\mathbb L(x_{i+1},c_0)$ 
at each $x_{i+1}$. The following theorem can be proved by using Lemma~\ref{l2}.
\begin{theorem}~\cite{CF83, CZ10}
 Any point $x_0$ generated by the smoothing gradient algorithm is a clarkr stationary point of 
$\mathbb L$ at $c_0\in\mathbb R$. 
\label{T100}
\end{theorem}
A clarkr stationary point is a solution of $\displaystyle\min_x\mathbb L(x,c_0)$~\cite{CF83}. Hence 
the smoothing gradient algorithm converges to a solution of $\displaystyle\min_x \mathbb L(x,c_0)$.

\subsection{Convergence of the Root Finding Method:}
\vspace{2mm}

\begin{theorem}
Let $\{y_i\}$, $\{\mu_i\}$ and $\{c_i\}$ be the sequences generated by the root finding iterative 
method. As $\{c_i\}$ converges to a root $c^*_0$ of the equation $\psi(c_0)=1$ the following two 
conditions will hold,
\begin{enumerate}
 \item $\displaystyle\lim_{i\rightarrow\infty} \mu_i=0$ and
 \item $\displaystyle\lim_{i\rightarrow\infty}||\triangledown\tilde{\mathbb L}(y_{i+1},\mu_i,c_0)|| 
= 0.$
\end{enumerate}
\end{theorem}

\begin{proof} We use Lemmma~\ref{l2} to prove this theorem.
 \begin{enumerate}
  \item Let at the $i$-th step of the root finding iterative method we choose 
$\mu_{i0}=\frac{1}{i}$ 
as initial guess for $\mu_0$ to start the smoothing gradient algorithm. At the end of the iteration 
let we obtain $\mu_i$ as the final value of $\mu$. Then $\mu_i\leq\mu_{i0} = \frac{1}{i}$. Hence 
$\displaystyle\lim_{i\rightarrow\infty}\mu_i=0$ follows since 
$\displaystyle\lim_{i\rightarrow\infty}\mu_{i0}=0$.
  \item From condition~$(2)$ of Lemma~\ref{l2} we get, at the $i$-th step of the root finding 
method $||\triangledown\tilde{\mathbb L}$ $(y_{i+1},\mu_i,c_0)||$ can be made less than 
$\frac{1}{i}$ by increasing the number of iteration in the smoothing gradient algorithm. Hence the 
proof follows similarly as $(1)$.
 \end{enumerate}
\end{proof}

Using Theorem~\ref{T100} we immediately get the following theorem. 
\begin{theorem}
As $\{c_i\}$ converges to a root $c^*_0$ of the equation $\psi(c_0)=1$ the iterative root finding 
method converges to a clarkr stationary point of the non-smooth function $\mathbb L(x,c_0)$. 
\end{theorem}

\subsection{Some Instances showing the performance of the proposed method}

In our work we consider different networks with noisy distances having noise up to $15\%$ over the 
exact distances between adjacent pair of nodes. We choose a $10 \times 10$ square region as the 
field of interest (we don't specify any unit for this distances since units can be chosen as 
required in the relevant field). We randomly deploy the networks having within this region and 
select a random set of vertex pairs as the edge set. The exact distances between adjacent pairs of 
nodes are measured and recorded. The maximum distance possible between such pairs of nodes is called 
diameter $(=10\sqrt 2)$ of the square region. To maintain a $15\%$ noise in the distance 
measurement, we choose a random number from $[-10\sqrt 2\times 0.15, 10\sqrt 2\times 0.15] \approx 
[-2,2]$ corresponding to each edge and add it to its exact distance as error. Under this setting, 
the estimated positions of nodes for different networks obtained from our algorithm is compared with 
their original positions.

\begin{example}
We consider a network consisting of $10$ nodes and $28$ communication links with weights of the 
links as the noisy distances. The original and estimated positions of node in the network obtained 
from our algorithm are recorded in Table~\ref{tab1}. 

\begin{table}[H]
\centering
\begin{tabular}{|c|c|c|}
\hline
Vertices & $(x, y)$(original) & $(x, y)$(estimated) \\
\hline
$1$ & $(0, 0)$    &    $(0,0)$                      \\
\hline
$2$ & $(4.0122, -0.0000 )$ &  $(4.1359, 0.0000)$    \\
\hline
$3$ & $(3.8810, -2.4025)$ &  $(4.0522, -2.3175)$    \\
\hline
$4$ & $(6.1459, -1.8400)$ &  $(6.3297, -1.6593)$    \\
\hline
$5$ & $(7.9481, -0.3167)$ &  $(7.9421, 0.1764)$     \\
\hline
$6$ & $(2.1969, -0.5585)$ &  $(2.3947, -0.6566)$    \\
\hline
$7$ & $(6.1260, 5.7528)$ &  $(0.0901, 2.5287)$      \\
\hline
$8$ & $(6.8309, 6.0573)$ &  $(6.6231, 6.2059)$      \\
\hline
$9$ & $(3.9878, 4.2560)$ &  $(3.7412, 4.3045)$      \\
\hline
$10$ & $(4.2515, 2.2306)$ &  $(4.0891, 2.3111)$     \\
\hline
\end{tabular}
\vspace{2mm}
\caption{The original and estimated node positions of an arbitrary network}
\label{tab1}
\end{table}

The Figure~\ref{tab2}(a) and Figure~\ref{tab2}(b) represents two different realizations of the 
network among which Figure~\ref{tab2}(a) corresponds to the node positions given in 
Table~\ref{tab1}. The network is not uniquely localizable since the node $7$ has degree $2$. As 
shown in Figure~\ref{tab2} the node $7$ may be assigned at least two different positions in any 
particular assignment of positions for the other nodes of the network in the plane. The 
Figure~\ref{tab3} represents the estimated node positions for the network using root finding method. 
The estimated node positions of the network (Figure~\ref{tab3}) are within small neighborhood of the 
realization of the network given in Figure~\ref{tab2}(b). 
\begin{figure}[ht]
\centering
 \includegraphics[width = 2.65in]{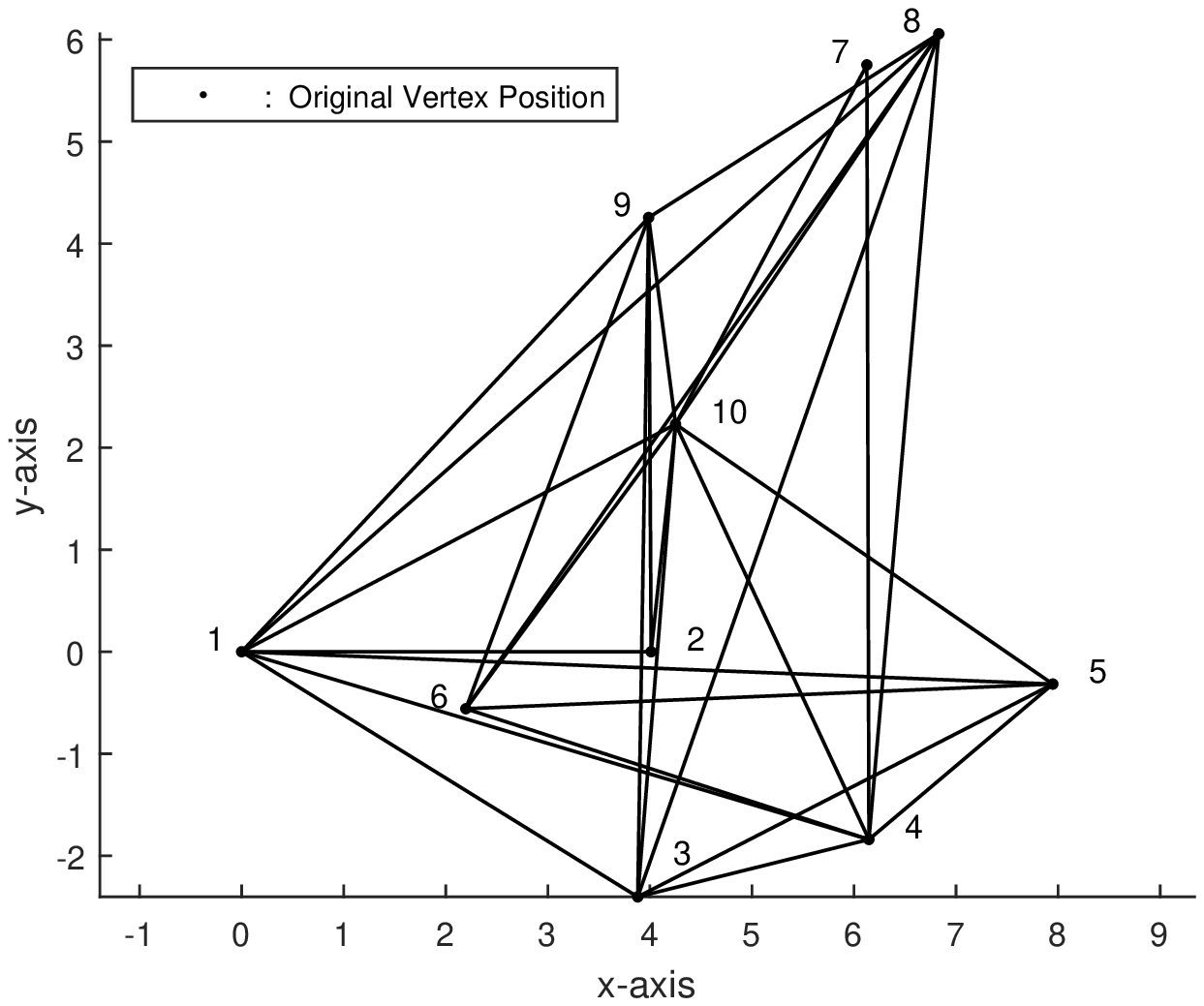}\hspace{.1mm}\includegraphics[width=2.65in]{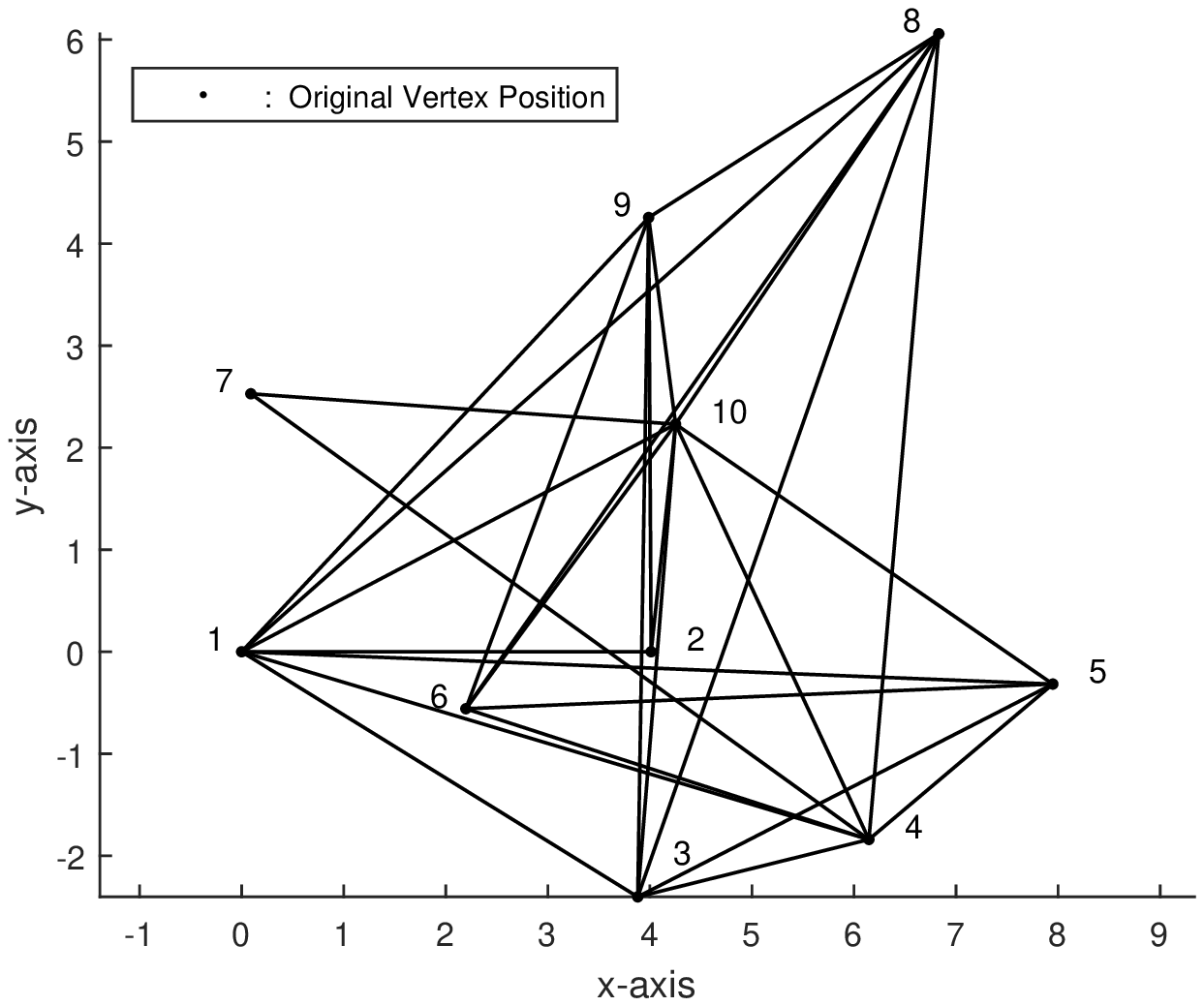}
 \caption{(a)\hspace{6.1cm}(b)}
 \label{tab2}
\end{figure}
\begin{figure}[ht]
  \centering
  \includegraphics[width = 2.65in]{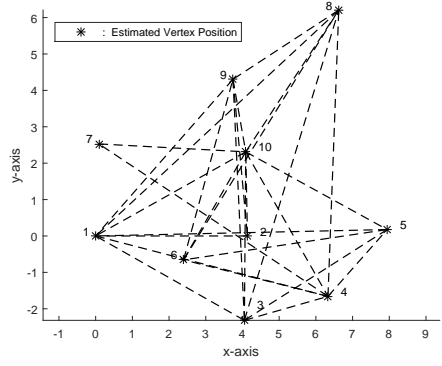}
  \caption{The estimated node positions in the network in Figure~\ref{tab2}}
\label{tab3}
\end{figure}
\end{example}

If a network is not uniquely localizable the associated localization problem has multiple solutions 
satisfying the distance constraints. The above measurement leads an intuition that the proposed 
method gives good estimation for node positions. 

\begin{example}
In Table~\ref{tab10} we are showing the original and estimated node positions for an uniquely 
localizable network. Figure~\ref{fig:15_32un} represents the associated networks along with their 
communication links.
\begin{table}[ht]
\centering
\begin{tabular}{|c|c|c|}
\hline
Vertices & $(x, y)$(original) & $(x, y)$(estimated) \\
\hline
$1$     & $(0,0)$      &   $(0,0)$      \\ 
\hline
$2$     & $(3.6881,0.0000)$ &   $(3.7735,0.0000)$ \\
\hline
$3$     & $(1.8213,0.2799)$ &   $(1.7589,0.3499)$ \\
\hline
$4$     & $(3.6192,2.4434)$ &   $(3.4261,2.5843)$ \\
\hline
$5$     & $(2.5173,0.8943)$ &   $(2.2052,0.1595)$ \\
\hline
$6$     & $(3.4462,2.4396)$ &   $(3.3709,2.2453)$ \\
\hline
$7$     & $(5.6695,4.4681)$ &   $(5.5768,4.5106)$ \\
\hline
$8$     & $(0.6039,5.6479)$ &   $(0.6087,5.6920)$ \\
\hline
$9$     & $(2.3696,3.3278)$ &   $(2.3902,3.3853)$ \\
\hline
$10$    & $(3.4283,0.1664)$ &   $(3.7578,0.0090)$ \\
\hline
$11$    & $(3.9853,4.5081)$ &   $(4.1825,4.3951)$ \\
\hline
$12$    & $(3.8426,1.3047)$ &   $(3.8101,1.2018)$ \\
\hline
$13$    & $(1.4380,5.3479)$ &   $(1.1181,5.7930)$ \\
\hline
$14$    & $(2.3117,1.9807)$ &   $(2.2834,2.0298)$ \\
\hline
$15$    & $(4.5512,2.6738)$ &   $(4.5605,2.7180)$  \\
\hline
\end{tabular}
\vspace{1mm}
\caption{The original and estimated node positions of an uniquely localizable network}
\label{tab10}
\end{table}
\begin{figure}[ht]
 \centering 
 \includegraphics[width=2.7in]{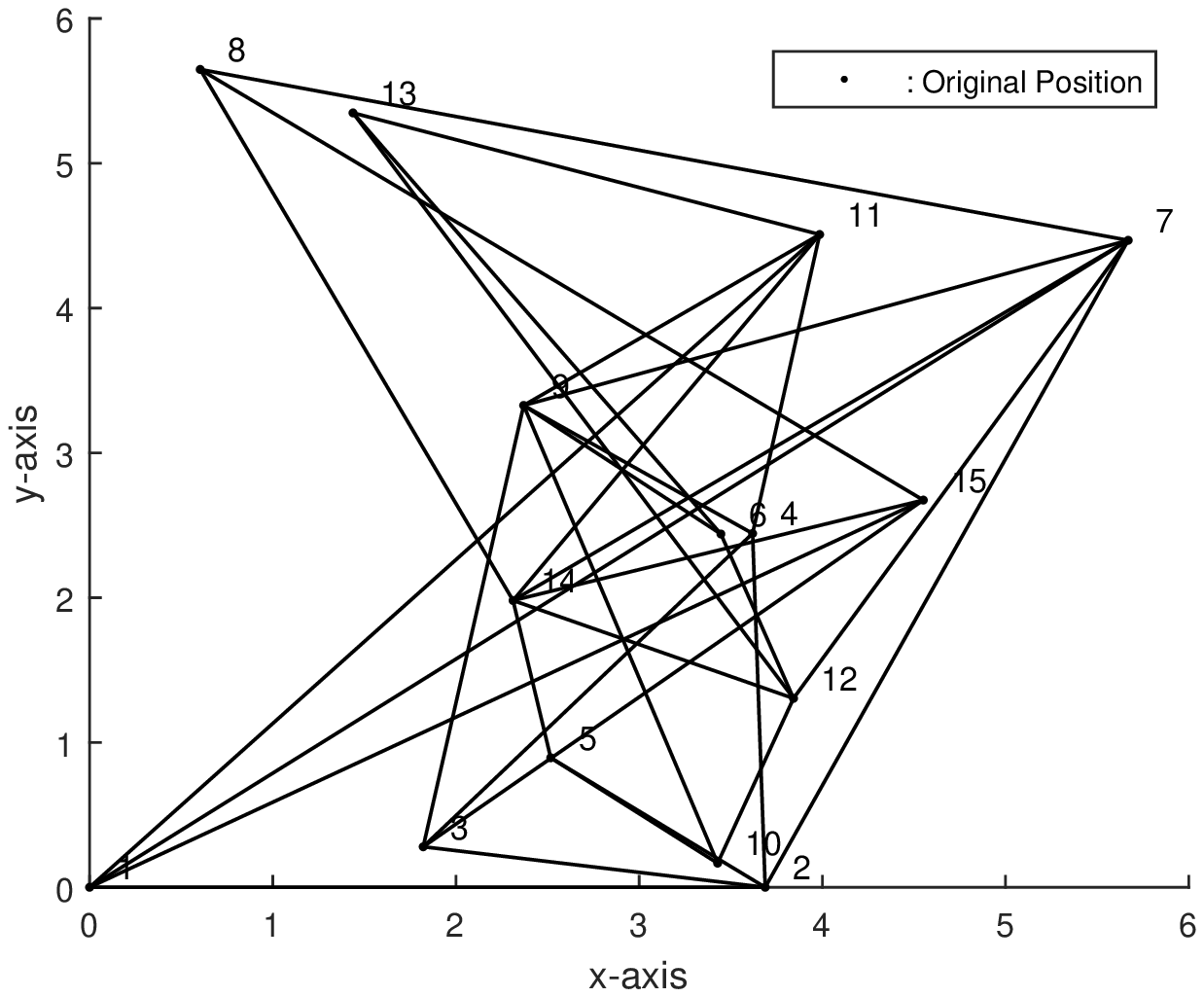}\includegraphics[width=2.7in]{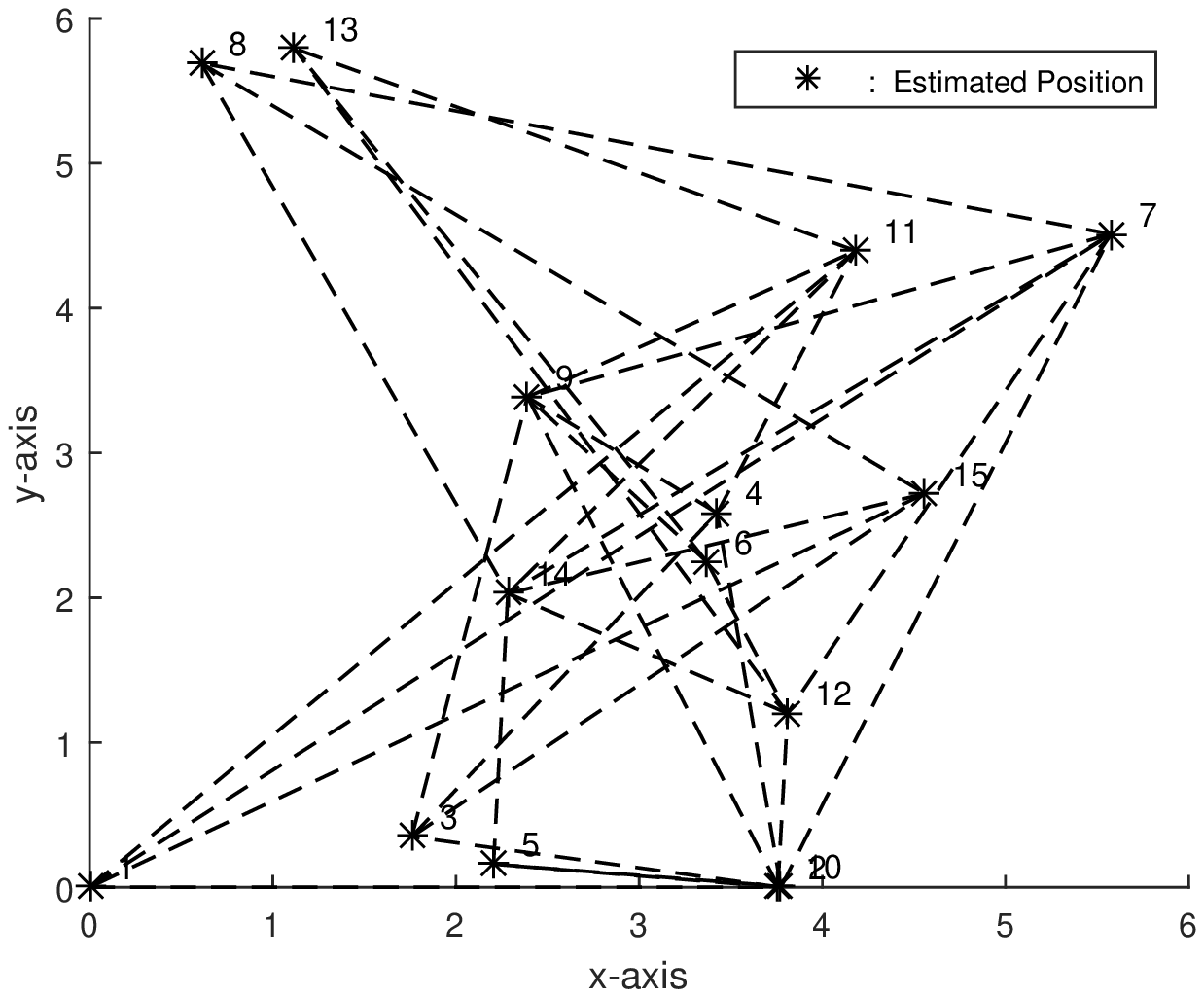}
 \caption{The original and estimated node positions of an uniquely localizable network}
 \label{fig:15_32un}
\end{figure}

We have seen that the estimated node positions are very close to the unique solution of the 
localization problem.
\end{example}

\textbf{Observations~1.} Either a network is uniquely localizable or not the root finding method 
gives a good estimation for the node positions. In case the localization problem has multiple 
solutions, the method gives an estimation which is very close to one solution in the set of all
solutions of the problem. 

If we increase the number of links in a random network keeping the number of nodes fixed then the 
probability that the network become uniquely localizable increases. The accuracy in position 
estimation may be better for networks having more edges. In the following we give some instances 
which show that as the number of edges increases in networks the position estimations become more 
accurate.

\subsection{Performance of the proposed method for increasing number of edges in a network}
Let $\aleph$ be a network having $n$ nodes and $e$ communication links. We define the 
\textit{network-density} $\rho(\aleph)$ as the ratio of $e$ to the maximum possible number of edges 
$n(n-1)/2$ (i.e., $\rho(\aleph)=2*e/n(n-1)$). Any uniquely localizable network $\aleph$ with $n$ 
number of nodes lying in the plane has at least $2n-2$ edges~\cite{JJ05} (e.g., a \textit{Laman} 
graph with an additional edge~\cite{L70}, a \textit{cycle bridge}~\cite{SM13sp}). Thus if $\aleph$ 
is uniquely localizable $\rho(\aleph)\geq 2(2n-2)/n(n-1)=4/n$. It may be noted that 
$\rho(\aleph)\geq 4/n$ does not imply the unique localizability of a network. In the following 
examples we have shown that the accuracy in position estimation increases with increasing network 
densities for randomly deployed networks as well the uniquely localizable networks.

\paragraph{\textbf{Some arbitrarily networks with different densities}}
In Figure~\ref{fig:15_32n}, Figure~\ref{fig:15_45n} and Figure~\ref{fig:15_55n} we plot the 
original 
and estimated node positions for randomly deployed networks for three different network-densities 
$0.30, 0.43$ and $0.52$. A small circle denotes the original position of a node and a star-mark 
denotes the computed position of a node. The error offset between the exact and estimated positions 
for individual nodes are indicated by lines. In these figures as the network-density increases the 
estimations for node locations become more accurate. It may be noted that large fluctuations occur 
due to flip ambiguities. Such fluctuations are also removed along with increasing densities.
\begin{figure}[ht]
\begin{minipage}{.47\textwidth}
 \centering
 \includegraphics[width = 2.58in]{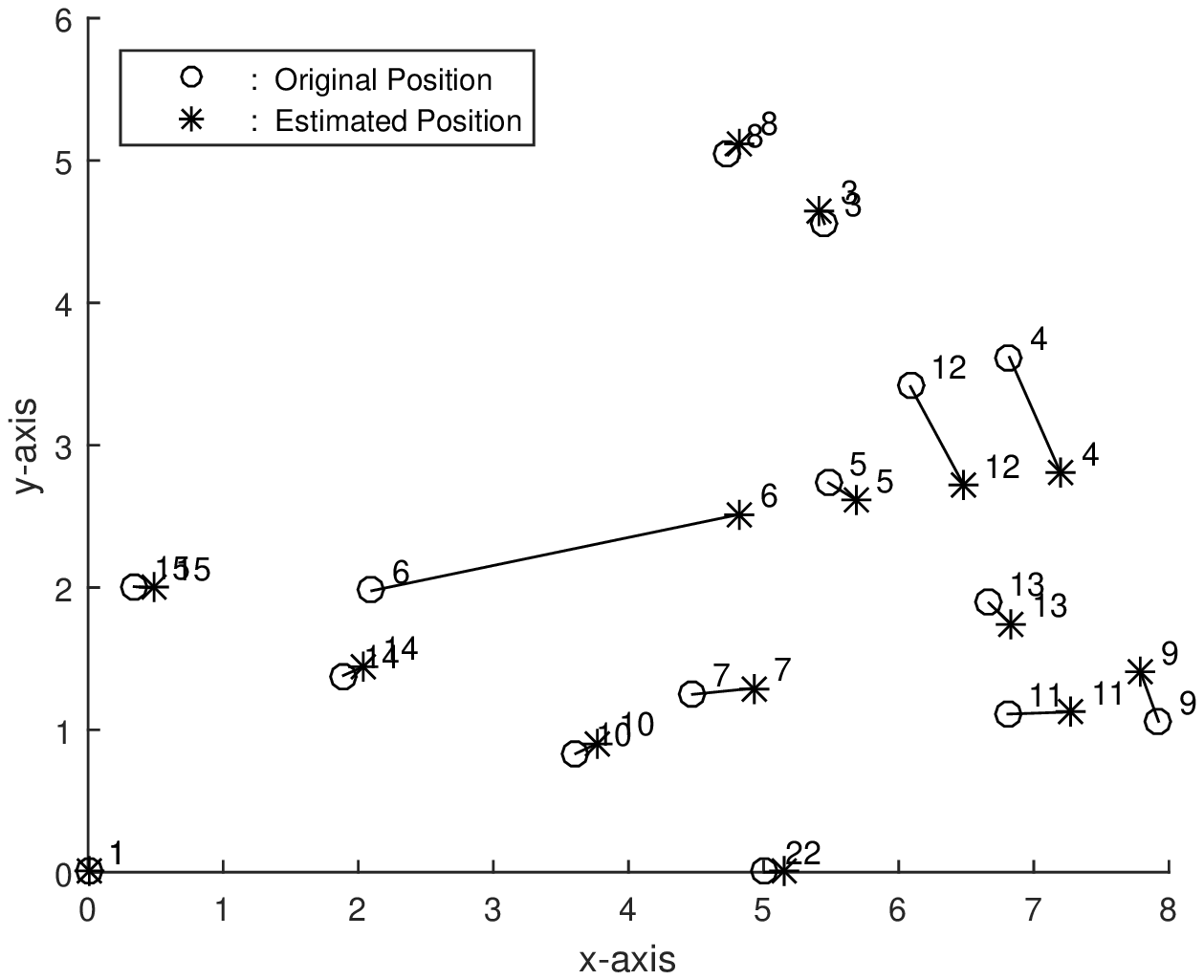}
 \caption{Random Network with density = 0.30}
 \label{fig:15_32n}
\end{minipage}
\begin{minipage}{.47\textwidth}
 \centering
 \includegraphics[width = 2.58in]{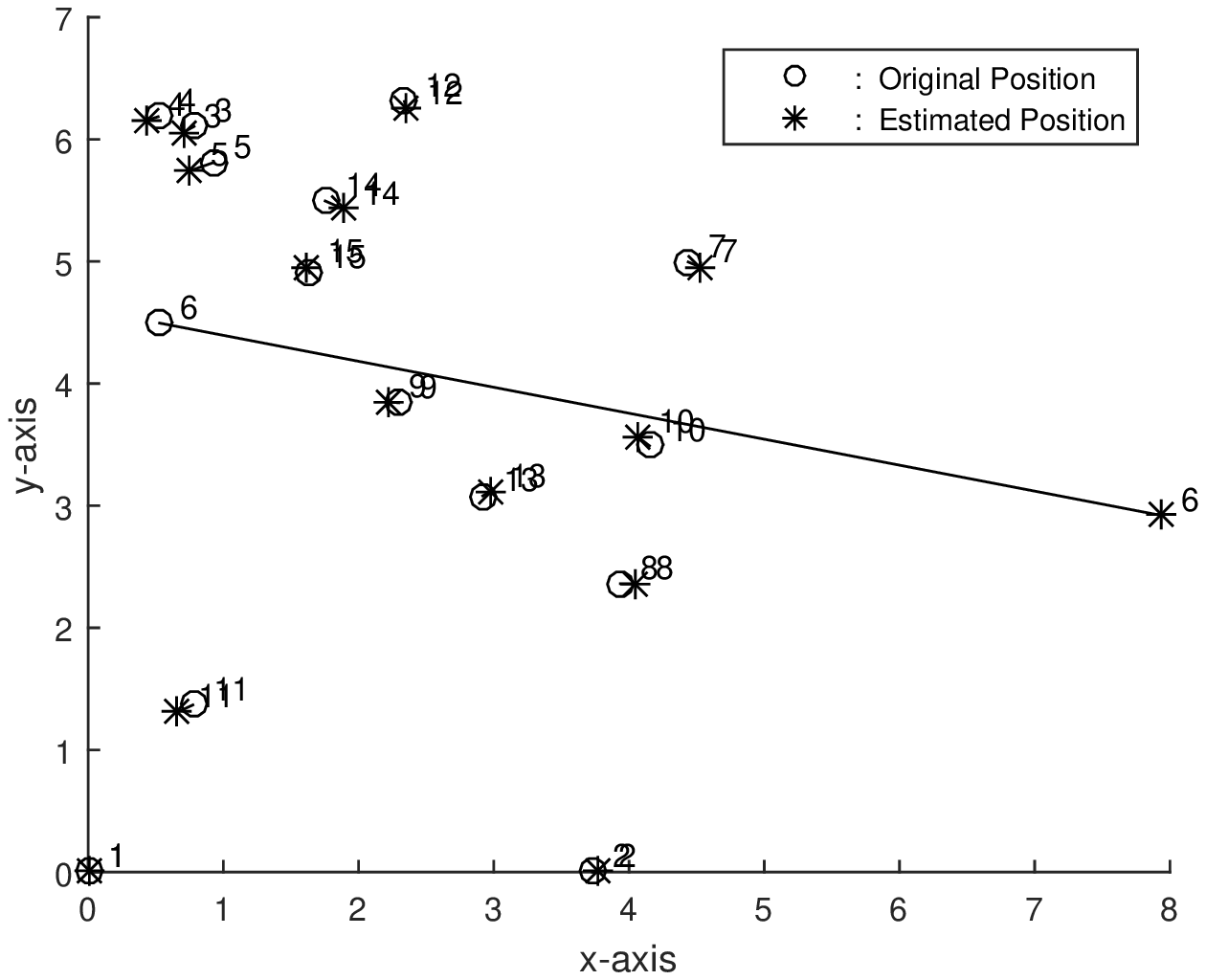}
 \caption{Random Network with density = 0.43} 
 \label{fig:15_45n}
\end{minipage}
\end{figure}
\begin{figure}[ht]
  \centering
  \includegraphics[width = 2.65in]{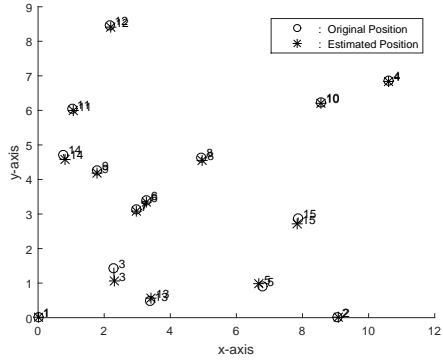}
  \caption{Random Network with density = 0.52}
  \label{fig:15_55n}
\end{figure}
\paragraph{\textbf{Some uniquely localizable networks with different network densities}}
Figure~\ref{un} and Figure~\ref{fig:15_55u} show scenarios of uniquely localizable networks with 
three 
network-densities $0.30,~0.43$ and $0.52$. For each network-density, the estimated node positions 
for the underlying network are very close to their respective original positions. 
\begin{figure}[ht]
 \begin{minipage}{.47\textwidth}
 \centering
  \includegraphics[width = 2.55in]{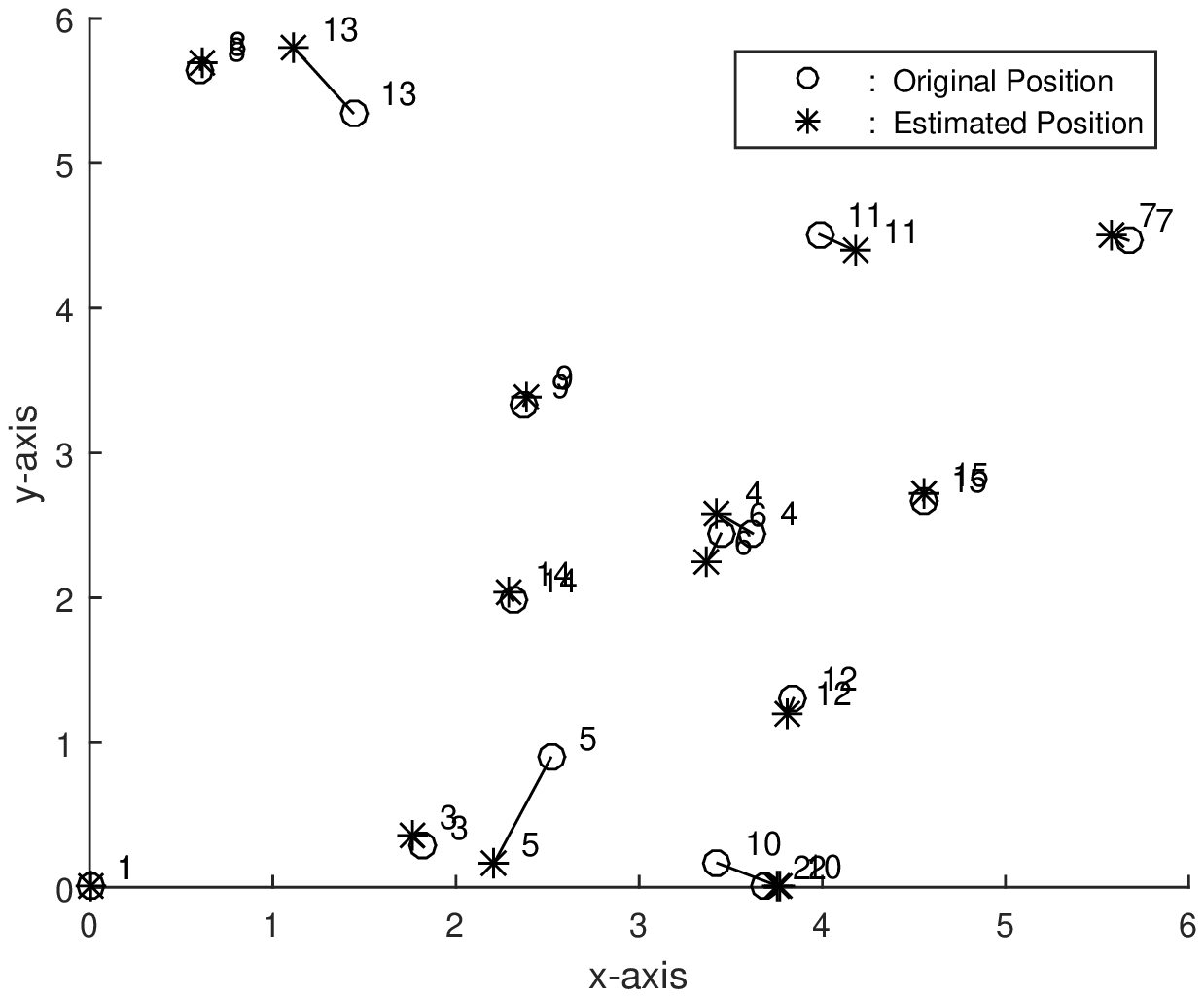}
 \end{minipage}
 \begin{minipage}{.47\textwidth}
  \centering
  \includegraphics[width = 2.55in]{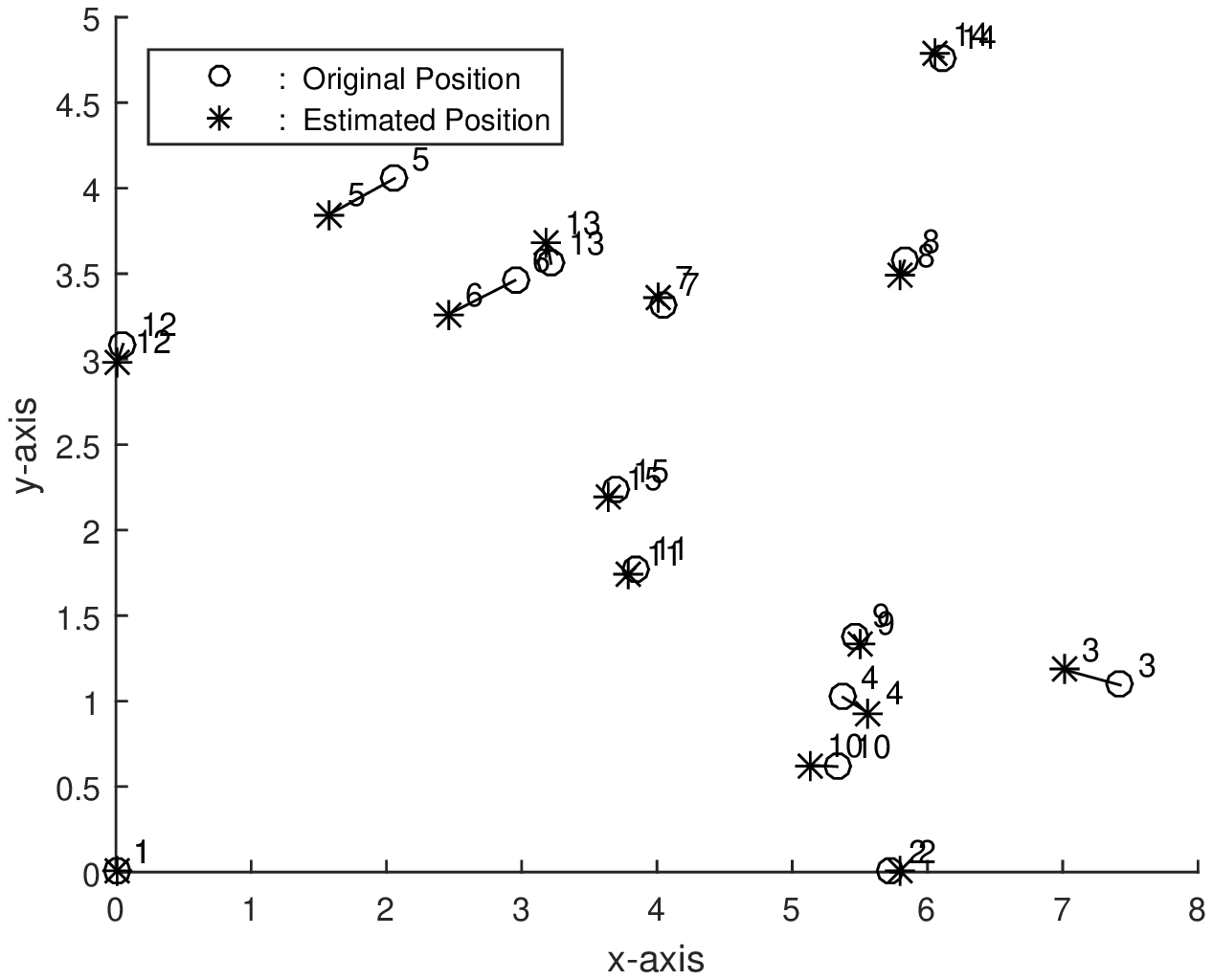}
 \end{minipage}
 \caption{Original and estimated node positions for uniquely localizable network with network 
densities 0.30 and 0.43 from left to right}
\label{un}
\end{figure}
\begin{figure}[tbhp]
 \centering
 \includegraphics[width =2.55in]{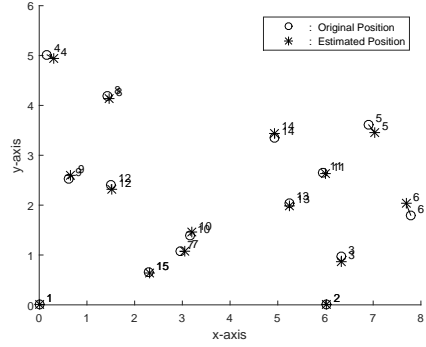}
 \caption{Original and estimated node positions for uniquely localizable network with network 
 density = 0.52}
 \label{fig:15_55u}
\end{figure}
Again this difference between the uniquely localizable networks and randomly chosen networks occurs 
due to flip of vertices in randomly chosen networks. However from Figure~\ref{un} and 
Figure~\ref{fig:15_55u} it is clear that the proposed algorithm works well. 
\vspace{2mm}

\textbf{Observation2.} Either the network is uniquely localizable or not the amount of errors are 
decreasing with increasing network densities.
\vspace{2mm}

Using an extensive number of networks, we shortly see that our claim is satisfied for a large 
number of networks.

\section{Error analysis of the root finding method}
\label{S6}
We have seen in Section~\ref{S5} that for a root $c_0=c^*_0$ of the function $\psi(c_0)-1$, the 
node 
positions for the underlying network is given by $x^*$ corresponding to which $\psi(c^*_0)=\mathbb 
L(x^*,c^*_0)=1.$ To analyze the performance of the location estimation different researchers have 
considered different error metrics~\cite{PKY8,DPG1}. In this work, we consider two different error 
metrics which are simple and standard. The performance of the root finding method is characterized 
with respect to these metrics.

\subsection{Error metrics}
\begin{enumerate}
 \item \textit{Mean Error}: The average of the differences between the estimated node positions (say
$x'_i$) and the original node position (say $x_i$) for all nodes is the \textit{mean error} of the 
position estimation, i.e., if a network has $n$ nodes then
$$Mean~Error=\frac{1}{n}\sum_i||x_i-x'_i||.$$
The mean error gives an overall idea about the accuracy of position estimations of the nodes in a
network. If for a network the mean error is small enough then it can be said that the estimated node
positions are very close to their original positions.

 \item \textit{Maximum Error}: This is the maximum difference among all the differences between the
estimated node locations and the general node locations of a network, i.e.,
$$Max~Error=\max_i||x_i-x'_i||.$$
The worst fluctuation in node position estimation is measured by the maximum error in any case.
\end{enumerate}

We consider a large number of random networks for finding positions of nodes using the root finding 
method. For each network we compute the mean and maximum errors in position estimation not removing 
the flip ambiguities. The error analysis shows that for this large class of networks, both the 
average mean and average maximum errors are decreasing with the growing network densities and 
gradually diminish to zero. 

\subsection{Setting the Environment for executing our algorithm}
We select a rectangular region ($10~unit\times 10~unit$) in the plane as the field of interest. 
Networks are randomly deployed within this region. We construct the graphs underlying to each 
network where the vertices and edges of the graph are considered as the nodes and communication 
links in the network respectively. The upper and lower bounds of the exact distance for each 
communication link are decided with a random distance error chosen within $0\%$ to $15\%$. The 
edge-weighted graph $G=(V,E,d)$ is constructed by setting the bounds as edge weights of the graph. 
The reconstruction of the network (i.e., assigning positions to the nodes) is carried out as 
follows:

\begin{itemize}
 \item A region $K$ in $\mathbb R^{2n}$ is considered (as described in Section~\ref{S4}). As 
discussed earlier all the feasible solutions (up to congruence) of the network localization problem 
have to be included in $K$.
 \item In $\mathbb R^{+}$ (set of all positive real numbers) an interval is identified 
(Section~\ref{S4}) within which the function $\psi(c_0)-1$ must have a root.
 \item A point $x_0$ is randomly chosen from $K$ as the initial input for starting the 
\Call{SmoothingGradientAlgorithm}{}.
 \item The estimated node positions of each network is recorded and compared to the original node
positions to observe the accuracy of estimation.
\end{itemize}

\subsection{Performance of our algorithm}
We choose some network densities such that the underlying networks may be sparse or may be dense for 
these densities. For each network density, more than $500$ networks are considered and positions of 
nodes are estimated using the proposed technique. From the original and estimated node positions of 
each network both the \textit{mean errors} and \textit{maximum errors} are computed. The average of 
all the mean errors and the average maximum errors for the selected network densities are computed.

We give two different scenarios for both the average mean error and the average maximum error on the 
basis of the amount of noise we allow in the input distance information. In the table~\ref{t100} the 
average errors are recorded for random networks where we allow maximum $10\%$ noise in the distance 
information. The computed average errors are small in this case. The Figure~\ref{fig:ErrGrp1} 
associated to the table shows that, as the network density increases the average mean error as well 
the average maximum error decrease. We observe for higher network densities that the average mean 
errors are insignificant. Thus for randomly chosen networks having $10\%$ noise in the distance 
information our proposed method gives a satisfactory result. We reach to a similar conclusion for 
the networks (Table~\ref{t101} and the corresponding plot in Figure~\ref{fig:ErrGrp2}) where the 
average errors are computed for networks having $15\%$ noise in distance information.

\begin{figure}[ht]
\begin{minipage}{.45\textwidth}
\centering
\begin{tabular}{|c|c|c|}
\hline
$\rho(\aleph)$ & \textit{Mean error} & \textit{Maximum error}  \\
\hline
~~~~~$0.40$ & $1.4863$ &  $4.3133$  \\
\hline
~~~~~$0.44$ & $0.7368$ &  $2.3716$  \\
\hline
~~~~~$0.51$ & $0.3380$ &  $0.9376$  \\
\hline
~~~~~$0.58$ & $0.2856$ &  $0.8061$  \\
\hline
~~~~~$0.64$ & $0.1570$ &  $0.3563$  \\
\hline
~~~~~$0.71$ & $0.1301$ &  $0.3196$  \\
\hline
~~~~~$0.78$ & $0.0919$ &  $0.2008$  \\
\hline
~~~~~$0.84$ & $0.0820$ &  $0.1753$  \\
\hline
~~~~~$0.91$ & $0.0781$ &  $0.1730$ \\
\hline
\end{tabular}
\vspace{1mm}
\caption{Errors computed for networks having maximum $10\%$ noise in distance constraints}
\label{t100}
\end{minipage}
\begin{minipage}{.5\textwidth}
 \centering
 \includegraphics[width=2.75in]{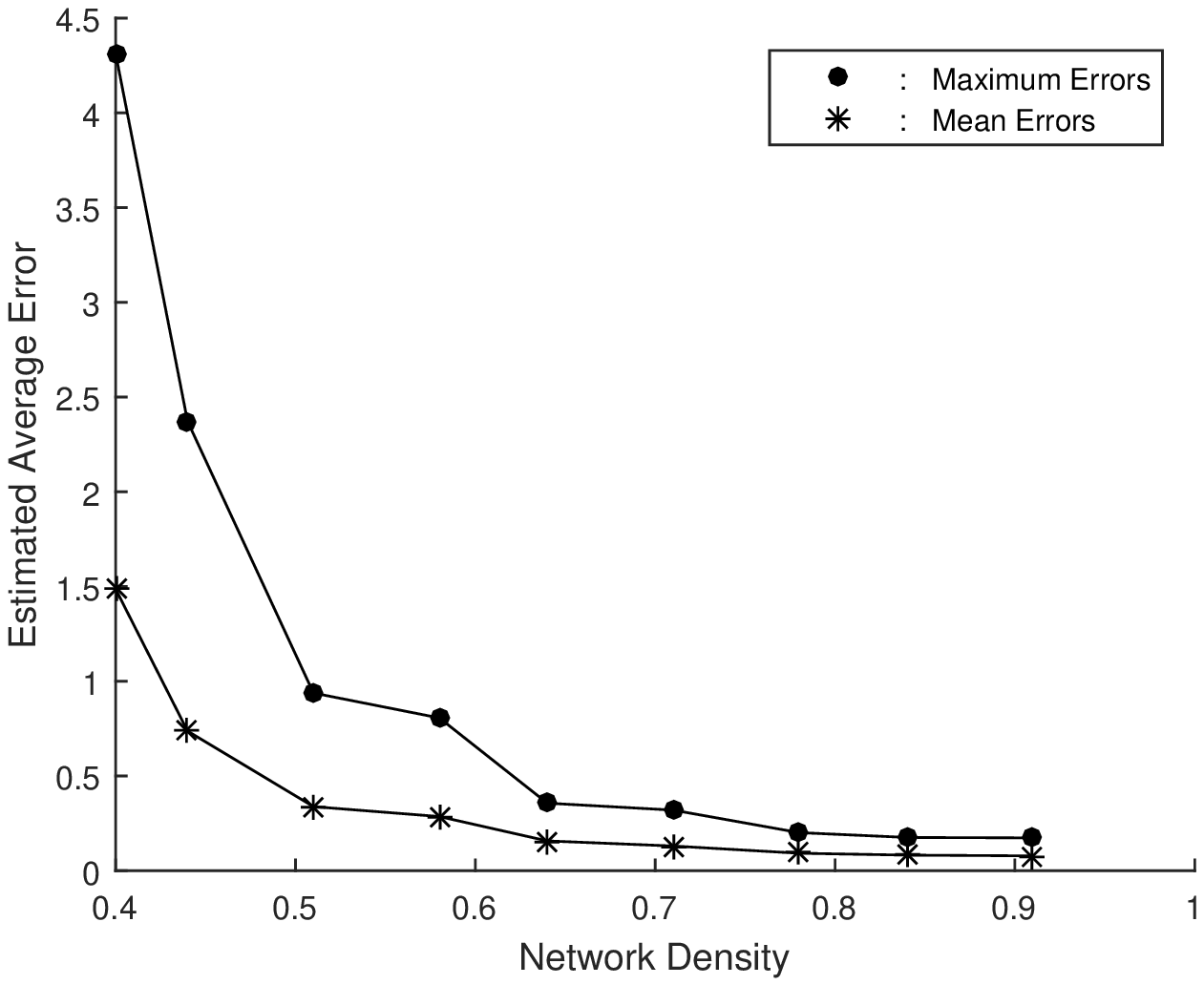}
 \caption{Polygon representing the values given in Table~\ref{t100}}
\label{fig:ErrGrp1}
\end{minipage}
\end{figure}

\begin{figure}[ht]
\begin{minipage}{.45\textwidth}
\centering
 \begin{tabular}{|c|c|c|}
    \hline
    $\rho(\aleph)$ & \textit{Mean error} & \textit{Maximum error}  \\
    \hline
    ~~$0.47$ & $0.8341$ &  $2.7959$ \\
    \hline
    ~~$0.51$ & $0.5798$ & $1.5613$  \\
    \hline
    ~~$0.56$ & $0.3578$ & $1.1638$  \\
    \hline
    ~~$0.60$ & $0.2672$ & $0.7396$  \\
    \hline
    ~~$0.64$ & $0.2659$ & $0.7014$  \\
    \hline
    ~~$0.69$ & $0.1810$ & $0.4455 $ \\
    \hline
    ~~$0.73$ & $0.1668$ & $0.3875$  \\
    \hline
    ~~$0.78$ & $0.1312$ & $0.3182$  \\
    \hline
    ~~$0.82$ & $0.1145$ & $0.2588$  \\
    \hline
 \end{tabular}
\vspace{1mm}
\caption{Errors computed for networks having maximum $15\%$ noise in distance constraints}
\label{t101}
\end{minipage}
\begin{minipage}{.5\textwidth}
  \centering
  \includegraphics[width=2.75in]{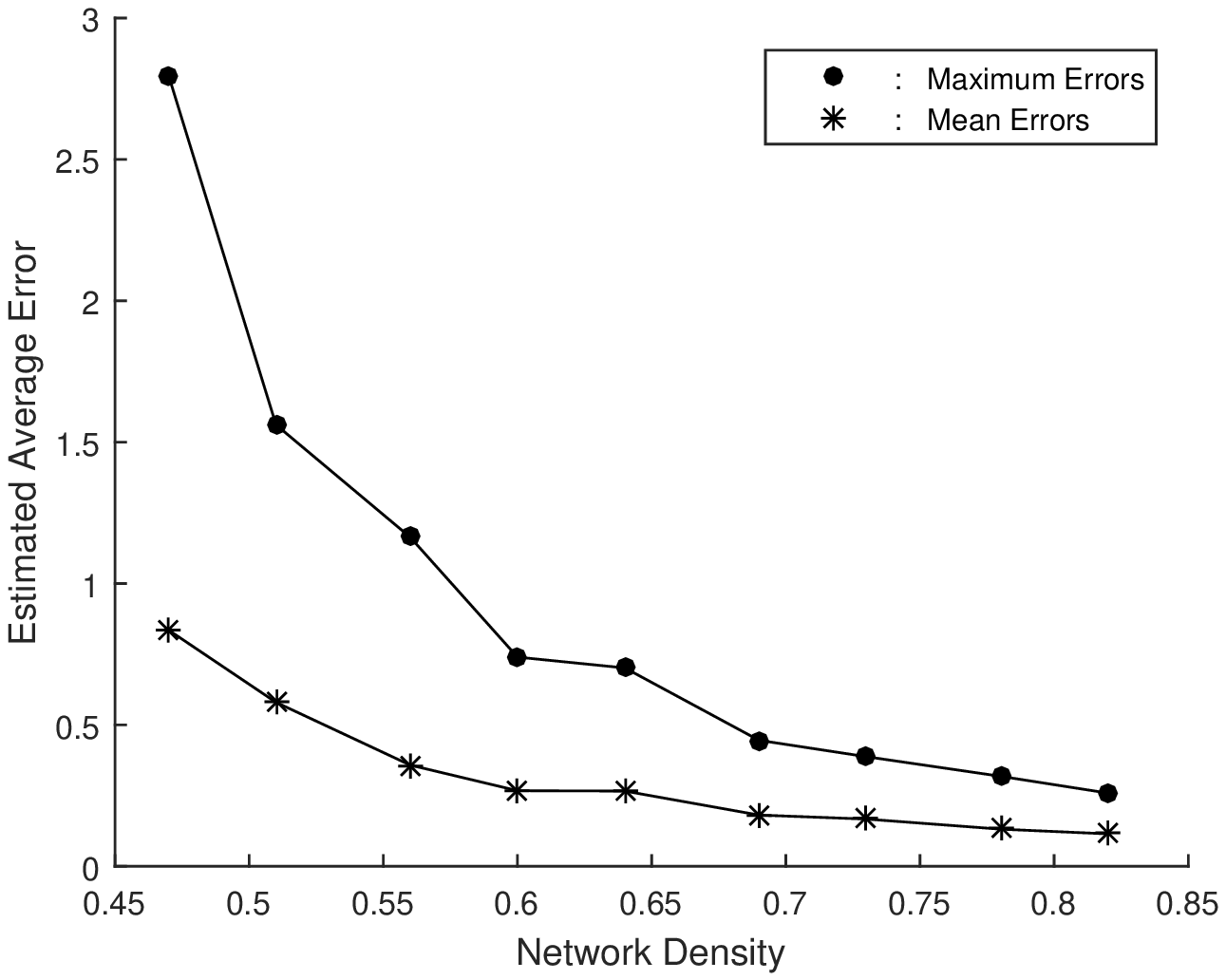}
  \caption{Polygon representing the values in Table~\ref{t101}}
  \label{fig:ErrGrp2}
\end{minipage}
\end{figure}

\subsection{The advantages of the proposed root finding method}
In the literature, there are several network localization techniques which claim to solve the 
localization problem. We have mentioned some of the popular localization techniques in 
Section~\ref{S1} which instead of solving the general network localization problem, have solved 
some modified version of it. Such modifications of the general network localization problem were 
done for implementing the existing convex optimization techniques like SDP, least square 
approximation etc. to solve the problem. For instances, in~\cite{DPG1} the non-convex distance 
constraints are eliminated and solved using SDP. In~\cite{PKY8}, the non-convex constrains are 
relaxed such that the modified problem becomes convex optimization problem.  In our work, we solve 
the general network localization problem (Problem~\ref{p1}) as the root finding problem 
(Problem~\ref{P6}). We enlist below some advantages of using the root finding method for solving 
the 
network localization problem.

\begin{enumerate}
 \item The root finding method uses nonlinear non-convex Lagrangian optimization technique to solve 
the general network localization problem which considers both the convex and non-convex distance 
constraints without any modification or relaxation.
 \item The actual feasible region of the network localization problem was modified in~\cite{PKY8} 
for implementing the SDP. Therefore the errors in position estimation was due to incorrect 
constraints. On the contrary, in the proposed method the constraints are used without any 
modification. Therefore the errors in position estimation occur due to the inherent limitations of 
numerical root finding technique. By increasing the number of iteration of the root computations, 
one can achieve a desired level of accuracy in the position estimation.
 \item The localization problems considered in ~\cite{PKY8} contains $0\%$ to $10\%$ noise over 
the actual distance measurements. The localization technique in~\cite{PKY8} requires further 
improvement if the noise raises to $10\%$ to $20\%$ in distances. In reality, the distances 
constraints in localization problem mostly involves larger errors than the $10\%$. In a test bed 
with $100$ nodes we experience at most $50\%$ noise in distance measurements. Though the examples 
exhibited in this paper contain maximum $20\%$ noise in distance measurements, the proposed 
method may be used for computing node positions if the percentage of noise increases.
 \item The convergence of root finding method is guaranteed to the unique solution of the network 
localization problem if the network is uniquely localizable. If the network is not uniquely 
localizable the method converges to one of the possible solutions of it. 
\end{enumerate}

\section{Conclusion}
\label{S7}
The growing functionality of wireless sensor networks (WSN) in different fields of real life 
applications requires the actual node positions as an information for properly monitoring over the 
detected events. The existing techniques for solving the network localization problem do not solve 
the network localization problem. Instead, some variants of the problem have been solved so far. To 
the best of our knowledge, this is the first approach for solving the general network localization 
problem in noisy environment. We take the advantage of the nonlinear Lagrangian function for 
non-convex constraints to transform the general network localization problem to a root finding 
problem and use a simple iterative method for finding the roots. The method is guaranteed to 
converge to a solution. The node positions of any network can be computed up to a desired level of 
accuracy in this method. The examples show that for randomly deployed networks the root finding 
method gives good estimations for node positions.

We have a plan to validate the method with a test bed consisting more than $100$ Arduino nodes. 
Our future target is to improve the root finding method in distributed techniques to enhance its 
performance in WSN.
\vspace{10mm}

\begin{center}
\textbf{ACKNOWLEDGMENT}
\end{center}
The first author would like to acknowledge Council of Scientific and Industrial Research (CSIR), 
Government of India, for giving the financial support in the form of Junior research fellowship to 
carry out this work.

\bibliographystyle{plain}
\bibliography{mybibfile}

\appendix

\section{Clarkr stationary point}
\label{app2}
In non-smooth analysis \textit{F.H. Clarkr} introduces the notion of generalized derivatives, 
normal and tangent cones~\cite{CF83} to grapple the non-smoothness together with the absence of 
convexity. Let $f:B\rightarrow\mathbb R\cup \infty$ be a non-smooth locally Lipschitz function 
where 
$B$ is a Banach Space. Let $B^*$ be the dual space of $X$. 
\begin{definition}~\cite{CF83}
The \textbf{generalized directional derivative} of $f$ at a point $\tilde{x}$ and in 
the direction of the vector $v$ is defined as
$$f^0(\tilde{x},v)=\displaystyle\limsup_{y\rightarrow x,t\downarrow 0}\frac{f(y+tv)-f(y)}{t}.$$
\end{definition}
Using theorems of functional analysis~\cite{} it may be said that there is a linear functional 
$\zeta:~B\rightarrow\mathbb R$ such that $f^0(x,v)\geq \zeta(v)$ for all $v\in B$. We adopt the 
convention of using $\langle\zeta,v\rangle$ for $\zeta(v)$. 
\begin{definition}
The \textbf{generalized gradient} at a point $\tilde{x}$ of $f$ is defined as
$$\partial f(\tilde{x}):=\{\zeta\in B^*:f^0(\tilde{x},v)\geq\langle\zeta,v\rangle\}$$
for all $v\in B$.
\end{definition}
\textit{Remark:} The function $f$ is \textbf{strictly differntiable} at a point $\tilde x$ 
if $\partial f(\tilde{x})$ is a singleton set, i.e., $\partial f(\tilde{x})=\{\triangledown 
f(\tilde{x})$\}.
\begin{theorem}
 $\partial f(\tilde{x})$ can be proved to be non-empty, convex and weak*-compact subset of $B^*$. 
\end{theorem}
The norm $||\zeta||_*$ in $B^*$ is defined as 
$$||\zeta||_*=\sup\Big\{\langle\zeta,v\rangle\Big|v\in B, ||v||\leq 1\Big\}.$$
\begin{definition}~\cite{MC06}
 A point $\tilde{x}$ is said to be a \textbf{clarkr stationary point} of $f$ if $0\in\partial 
f(\tilde x,v)$, i.e., $f^0(\tilde{x},v)\geq0$ for every direction vector $v$ in the clarkr tangent 
cone.
\end{definition}
\begin{definition}~\cite{CF83}
 The \textbf{clarkr tangent cone} to a point $\tilde{x}\in U\subset X$ consists of all the vectors 
$v$ of $X$ for which the following condition holds:

For every sequence $\{x_i\}$ in $U$ converging to $\tilde{x}$ and sequence $\{t_i\}$ $(t_i>0$ 
$\forall i)$ decreasing to $0$, there is a sequence $\{v_i\}$ in $X$ converging to $v$ such that 
$x_i+t_iv_i\in U$ for all $i$.
\end{definition}
The following theorems are useful for finding the clarkr stationary points of a non-smooth function.
\begin{theorem}~\cite{CF83}
 $\zeta\in\partial f(\tilde{x})$ if and only if $f^0(\tilde x, v)\geq\langle\zeta,v\rangle$ for all 
directions $v$ in the underlying space.
\end{theorem}

\begin{theorem}~\cite{CF83}
 Let $\triangledown f(\tilde x)$ exist at a point $\tilde x$ near which $f$ is Lipschitz 
continuous. Then $\triangledown f(\tilde x)\in\partial f(\tilde{x})$ 
\end{theorem}

\end{document}